\documentclass{article}
\usepackage[utf8]{inputenc}
\usepackage{hyperref,xcolor}
\usepackage{caption,subcaption,longtable}
\usepackage{graphics, setspace}
\usepackage{array,multirow}
\usepackage{tikz,float}
\usepackage{amsmath,amsthm,amssymb}
\usepackage{enumerate,enumitem,listings}
\usepackage[T1]{fontenc}

\newcommand{\Del}[1]{\Delta_{#1}( A_{n}\;\vert\;R)}
\newcommand{\W}[1]{\mathcal{W}(n,#1)}
\newcommand{\vect}[1]{\underline{\boldsymbol{#1}}}
\newcommand{\norm}[1]{\vert \vert \underline{\boldsymbol{#1}}\vert \vert }
\newcommand{\bW}[1]{\mathcal{W}\bigg(n,#1\bigg)}
\newcommand{\pres}[2]{\langle #1\;\vert\;#2\rangle}

\newcommand{\aas}[1]{a.a.s.$(#1)$}

\newtheorem{theoremalph}{Theorem}

\newtheorem{theorem}{Theorem}[section]
\newtheorem{lemma}[theorem]{Lemma}

\newtheorem*{remarkit*}{Remark}
\newtheorem*{theorem*}{Theorem}
\newtheorem*{lemma*}{Lemma}
\newtheorem*{proposition*} {Proposition}

\theoremstyle{definition}
\newtheorem{definition}[theorem]{Definition}
\newtheorem{remark}[theorem]{Remark}
\newtheorem*{remark*}{Remark}

\setlist[enumerate]{align=left}

\setlist[enumerate]{align=left}

\title{Property (T) in density-type models of random groups}
\author{Calum J. Ashcroft}
\date{\vspace{-2em}}

\begin{document}
\maketitle
	\begin{abstract}
We study Property (T) in the $\Gamma(n,k,d)$ model of random groups: as $k$ tends to infinity this gives the Gromov density model, introduced in \cite{gromovasymptotic}. We provide bounds for Property (T) in the $k$-angular model of random groups, i.e. the $
\Gamma (n,k,d)$ model where $k$ is fixed and $n$ tends to infinity. We also prove that for $d>1\slash 3$, a random group in the $\Gamma(n,k,d)$ model has Property (T) with probability tending to $1$ as $k$ tends to infinity, strengthening the results of \.{Z}uk and Kotowski--Kotowski, who consider only groups in the $\Gamma (n,3k,d)$ model. 
	\end{abstract}

 
\section{Introduction}
\subsection{Property (T) in random groups}
Gromov proposed two models of random groups in \cite{gromovasymptotic} to study the notion of a `generic' finitely presented group. There is some ambiguity in the literature between the two models, and so we provide the full definitions here.

Fix $n\geq 2, k\geq 3$, and $0<d<1$. The (\emph{strict}) $(n,k,d)$ \emph{model} is obtained as followed. Let $A_{n}=\{a_{1},\hdots, a_{n}\}$, and let $F_{n}:=\mathbb{F}(A_{n})$ be the free group generated by $A_{n}$. Let $\mathcal{C}(n,k)$ be the set of cyclically reduced words of length $k$ in $F_{n}$ (so that $\mathcal{C}(n,k)\approx (2n-1)^{k})$. Uniformly randomly select a set $R\subseteq \mathcal{C}(n,k)s$ of size $\vert R\vert =(2n-1)^{kd}$, and let $\Gamma:=\pres{A_{n}}{R}$. We call $\Gamma$ a \emph{random group} in the (\emph{strict}) $(n,k,d)$ \emph{model}, and write $\Gamma\sim \Gamma(n,k,d)$.

If we keep $n$ fixed and let $k$ tend to infinity, then we obtain the \emph{Gromov density model}, as introduced in \cite{gromovasymptotic}, whereas if we fix $k$ and let $n$ tend to infinity we obtain the \emph{$k$-angular model}, as introduced in \cite{ashcroftrandom}. The $k$-angular model was first studied for $k=3$ (the \emph{triangular model}) by \.{Z}uk in \cite{Zuk} and for $k=4$ (the \emph{square model}) by Odrzyg{\'o}{\'z}d{\'z} in \cite{odrsquaremodel}.

The \emph{lax} $(n,k,d,f)$ \emph{model} is obtained via the following procedure. Let \\$\mathcal{C}(n,k,f)$ be the set of cyclically reduced words of length between $k-f(k)$ and $k+f(k)$ in $F_{n}$, where $f(k)=o(k)$. Uniformly randomly select a set $R\subseteq \mathcal{C}(n,k,f)$ of size $\vert R\vert =(2n-1)^{kd}$, and let $\Gamma:=\pres{A_{n}}{R}$. We call $\Gamma$ a \emph{random group} in the \emph{lax $(n,k,d,f)$ model}, and write $\Gamma\sim \Gamma_{lax}(n,k,d,f)$.

 We first consider the case of the $k$-angular model. It is a seminal theorem of \.{Z}uk \cite{Zuk} (c.f. \cite{kotowski}) that for $d>1\slash 3$  a random group in the triangular model has Property (T) with probability tending to $1$. As observed in \cite{odrzygozdz2019bent} the case of $k$ divisible by $3$ is easier, as we may use the work of \cite{Zuk} and \cite{kotowski} to observe Property (T) at densities greater than $1\slash 3$: see \cite{Montee_prop_t} for the proof that $3k$-angular has Property (T) for any $d>1\slash 3$. This idea was in fact extended in \cite{Montee_prop_t} to passing from Property (T) in $\Gamma (n,k,d)$ to $\Gamma (n,lk,d)$ for $l\geq 1$. For $k\geq 3$, let 
$$d_{k}:=\frac{k+(-k\mod 3)}{3k}.$$
Here, we take the convention that $-k\mod 3$ represents $(-k)\mod 3$, we will always write $-(k\mod 3)$ to represent the alternative. In particular,
$$d_{k}=\begin{cases}\frac{1}{3}\mbox{ if }k=0\mod 3,\\
\frac{k+2}{3k}\mbox{ if }k=1 \mod 3,\\
 \frac{k+1}{3k} \mbox{ if }k=2\mod 3.
\end{cases}$$

In the case that $k=0\mod 3,$ $d_{k}=1\slash 3$, which is known to be the sharp threshold for Property (T) in the cases that $k=3$ \cite{Zuk,kotowski} and $k=6$ \cite{odrzygozdz2019bent}. The remaining cases are not known to be sharp.
 
Below, we analyse Property (T) in the $k$-angular model. We believe this to be the first non-trivial result on Property (T) in any $k$-angular model for any $k\geq 5$ not divisible by $3$, and in fact provides a non-trivial range of densities where random $k$-angular groups are infinite and have Property (T) for each $k\geq 8$. There is currently no known density for random $k$-angular groups to be infinite with Property (T) for $k=4,5,7$.
\begin{theoremalph}\label{mainthm: property t k-angular model}
Let $k\geq 8$, let $d>d_{k}$, and let $\Gamma_{m}\sim \Gamma (m,k,d)$. Then $$\lim_{m\rightarrow\infty }\mathbb{P}(\Gamma_{m}\mbox{ has Property }(T))=1.$$
\end{theoremalph}

Secondly, we can consider the density model. Again, there is some ambiguity between the strict model and the lax model in the literature. Indeed, many cubulation results, such as those of \cite{Ollivier-Wise} and \cite{mackay_przytycki2015balanced}, refer to groups in the strict model, whilst results on Property (T) typically refer to groups in the lax model. In particular, the following result is due to \.{Z}uk \cite{Zuk} and Kotowski--Kotowski \cite{kotowski} (see \cite{antoniuktriangle} for finer analysis of $\Gamma (n,3,d)$ as $d\rightarrow 1\slash 3$). There is an alternative proof of the below in \cite[Corollary $12.7$]{DRUTU_Mackay}.
\begin{theorem*}\cite{Zuk,kotowski}
Fix $n\geq 2$, let $d>1\slash 3$, and let $\Gamma_{k}\sim \Gamma (n,3k,d)$. Then $$\lim_{k\rightarrow\infty }\mathbb{P}(\Gamma_{k}\mbox{ has Property }(T))=1.$$
\end{theorem*}
Note that the above results \emph{only apply to groups whose relator length is divisible by $3$}. However, this result has two important consequences: firstly it provides an infinite number of hyperbolic torsion free groups with Property (T), since such groups are torsion free with probability tending to $1$, and the Euler characteristic of such a group is dependent only on $k$ and $d$ \cite{olliviersharp}. Secondly, it proves that groups in the $\Gamma_{lax}(n,k,d,f)$ model have Property (T), using the following argument; see \cite[I.2.c]{Ollivier_Jan_Invitation}. This step is noted in \cite[p. 410]{kotowski}.
\begin{remarkit*}
Fix $n\geq 2$, let $d>0$, and let $k_{i}$ be a sequence of increasing integers such that $\vert k_{i+1}-k_{i}\vert$ is uniformly bounded. If

$$\lim_{k_{i}\rightarrow\infty }\mathbb{P}(\Gamma\sim \Gamma(n,k_{i},d)\mbox{ has Property (T)})=1,$$ then there exists a constant function $f$ such that for any $d'>d$, 

$$\lim_{l\rightarrow\infty }\mathbb{P}(\Gamma\sim \Gamma_{lax}(n,l,d',f)\mbox{ has Property }(T))=1.$$
\end{remarkit*}
\begin{proof}
Let $C=\max_{i}\vert k_{i+1}-k_{i}\vert$, and choose $f=f(l)$ such that $f(l)\geq C$. For each $l$ choose $k_{i(l)}$ such that $l+f(l)-C\leq k_{i(l)}\leq l+f(l)$. Then for sufficiently large $l$, and for $\Gamma_{l}=\langle A_{n}\;\vert\;R\rangle \sim\Gamma_{lax}(n,l,d',f)$, we see that for any $d<d''<d'$, with probability tending to $1$ as $l$ tends to infinity, $$\vert R\cap \mathcal{C}(n,k_{i(l)})\vert\geq (2n-1)^{d''k_{i(l)}}.$$ Hence, by choosing a random subset $R'\subseteq R\cap \mathcal{C}(n,k_{i(l)})$ of size $(2n-1)^{kd}$, and setting $\Gamma_{i(l)} ':=\langle A_{n}\;\vert\;R'\rangle,$ we see that there exists an epimorphism $\Gamma'_{i(l)} \twoheadrightarrow\Gamma_{l},$ and $\Gamma'_{i(l)}\sim\Gamma (n,k_{i(l)},d).$ Since $\Gamma_{i(l)}'$ has Property (T) with probability tending to $1$ as $i(l)$ tends to infinity, and Property (T) is preserved by epimorphisms, the result follows.
\end{proof}

However, we note that the question of Property (T) remains open for the strict model. If $\lim_{k\rightarrow\infty }\mathbb{P}(\Gamma\sim \Gamma(n,k,d)\mbox{ has Property (T)})=1,$ then we must also have that 
$\lim_{p_{i}\rightarrow\infty }\mathbb{P}(\Gamma\sim \Gamma(n,p_{i},d)\mbox{ has Property (T)})=1,$
where $p_{i}$ denotes the $i^{th}$ prime. Since the results of \cite{Zuk,kotowski} do not apply in this regime, we are inspired to further analyse the question of Property (T) for $\Gamma (n,k,d).$

We now briefly explain the approach taken by \cite{Zuk,kotowski} to prove their theorem. Firstly, one takes $n\geq 2$, $d>1\slash 3$, and considers $\Gamma_{m}\sim \Gamma (m,3,d)$. It can then be proved that 
$$\lim_{m\rightarrow\infty }\mathbb{P}(\Gamma_{m}\mbox{ has Property (T)})=1.$$

The proof of the above is very involved, and requires passing via an alternate model, the \emph{permutation model}: we omit the definition of this model as we do not require it.

One fixes $d'>d$, and finds for each $k$ an integer $m(k,n)$ and a surjection $\Gamma_{m(k,n)}\twoheadrightarrow \Gamma'_{k}$, where $\Gamma_{k}'\sim \Gamma (n,3k(m,n),d')$ (technically this is a surjection onto a finite index subgroup of $\Gamma_{k}'$). The result then follows by preservation of Property (T) under epimorphisms and taking finite index extensions. 

A natural approach to extend the results to the strict model using the techniques of \.{Z}uk and Kotowski--Kotowski would be to to fix $l\geq 3$, let $\Gamma_{(m,l)}\sim \Gamma (m,l,d)$, and consider $m\rightarrow\infty$. Then for each $n\geq 2$ and $k\geq 3$ find an integer $m(k,l,n)$ and $$\Gamma '_{k}\sim \Gamma (n,lk,d')$$ with $\Gamma_{(m(k,l,n),l)}\twoheadrightarrow\Gamma'_{k(m,n,l)}$, as in \cite{Montee_prop_t}. However, if we consider the model $\Gamma(n,p_{i},d')$, then we must have in the above that $lk=p_{k}$ where $p_{k}$ is the $k^{th}$ prime number, which necessarily forces $m(k,l,n)=n$, and therefore we cannot use statements of the form $\lim_{m\rightarrow\infty }\mathbb{P}(\Gamma_{m}\mbox{ has Property (T)}),$ as $m$ must be bounded.

To address this, we therefore must deal with the model $\Gamma(n,k,d)$ directly. The approach is to use the work of  Ballmann--\'{S}wi\k{a}tkowski \cite{Ballmann-Swiatkowski} and \.{Z}uk \cite{zuk1996} (c.f. \cite{Zuk}), in which a spectral condition for Property (T) was provided independently.
This will be used to provide an alternate criterion for Property (T) in terms of the first eigenvalue of a graph we define relative to $\Gamma$, $\Delta_{k}(\Gamma)$. A similar graph was used in the case of $k=0\mod 3$ by Drutu--Mackay \cite{DRUTU_Mackay}. The bulk of this text then analyses the eigenvalues of these random graphs.

The following completes the analysis of Property (T) in $\Gamma (n,k,d)$ for $d>1\slash 3.$
\begin{theoremalph}\label{mainthm: property t density model}
Let $n\geq 2$, $d>1\slash 3$, and let $\Gamma_{k}\sim \Gamma (n,k,d)$. Then: 
$$\lim_{k\rightarrow\infty }\mathbb{P}(\Gamma_{k}\mbox{ has Property }(T))=1.$$
\end{theoremalph}
Note that this immediately implies for any infinite sequence, $\{k_{i}\}_{i}$, of increasing positive integers, and $\Gamma_{i}\sim \Gamma (n,k_{i},d)$ that: 
$$\lim_{i\rightarrow\infty }\mathbb{P}(\Gamma_{i}\mbox{ has Property }(T)),$$
so that we immediately recover the results of \cite{Zuk, kotowski}.

We note that we could also consider the case of $d\rightarrow 1\slash 3$ in a manner similar to that of \cite{antoniuktriangle}. For $n\geq 2$, $k \geq 3$, and $0<p<1$, we can define the random group model $\Gamma_{p}(n,k,p)$: let $\Gamma =\langle A_{n}\;\vert\;R\rangle,$ where $R$ is obtained by adding each word in $\mathcal{C}(n,k)$ with probability $p$. Since Property (T) is an increasing property (one preserved by epimorphisms), it is easy to switch between $\Gamma_{p}(n,k,p)$ and $\Gamma (n,k,(2n-1)^{k}p)$ in a manner analogous to switching between the Erd\"os--R\'enyi random graph $G(m,p)$ and the random graph $G(m,M)$, since the number of relators in $R$ is $\vert R\vert =(1+o(1))(2n-1)^{k}p$ almost surely, for $p$ sufficiently large. In fact, we do analyse Property (T) in $\Gamma_{p}(n,k,p)$ in Theorem \ref{thm: Gp has Property (T)}, and then use this to prove Theorems \ref{mainthm: property t k-angular model} and \ref{mainthm: property t density model}. However, we believe that the notation and constants involved in the statement of Theorem \ref{thm: Gp has Property (T)} add unnecessary complexity to the statement of Theorem \ref{mainthm: property t density model}, and so we leave this to Section \ref{sec: Property (T) in random quotients of free groups}.

Indeed, random groups exhibit many interesting properties, depending on the density chosen. All of the following statements hold asymptotically almost surely, i.e. with probability tending to $1$. Firstly, a random group in the density model at density $d<1\slash 2$ is hyperbolic and torsion-free \cite{gromovasymptotic} (c.f. \cite{olliviersharp}). This argument also transfers to the $k$-angular model \cite{ashcroftrandom}: see \cite{odrsquaremodel} for the case of $k=4$, as well as a generalisation of the argument to a wider class of diagrams.
In the opposite direction to Property (T), there are many results known about the lack of Property (T) in various models of random groups. Groups in the density model are virtually special for $d<1\slash 6$ \cite{Ollivier-Wise} and contain a free codimension-$1$ subgroup for $d<5\slash 24$ \cite{mackay_przytycki2015balanced}. As observed in \cite{odrzygozdz2019bent} this implies that for any $k\geq 3$, a random group in the $k$-angular model at density $d<5\slash 24$ does not have Property (T). Groups in the triangular model are free at densities less than $1\slash 3$ \cite{antoniuktriangle}, groups in the square model are free at densities less than $1\slash 4$ \cite{odrsquaremodel}, and groups in the $k$-angular model are free for $d<1\slash k$ \cite{ashcroftrandom}. Furthermore, groups in the square model are virtually special for $d<1\slash 3$ \cite{duong,odrcubulatingsquare} and contain a codimension-$1$ subgroup for $d<3\slash 8$ \cite{odrzygozdz2019bent}. Finally, groups in the hexagonal model contain a codimension-$1$ subgroup for $d<1\slash 3$ and have Property (T) for $d>1\slash 3$ \cite{odrzygozdz2019bent}.

\subsection{Structure of the paper and some notation}
The idea of the proof is the following: for a finitely presented group $\Gamma$ we find a graph $\Delta (\Gamma)$, and using work of \cite{zuk1996}, \cite{Ballmann-Swiatkowski}, we prove that if $\lambda_{1}(\Delta(\Gamma))>1\slash 2$, then $\Gamma$ has Property (T). This graph loosely corresponds to the `link of depth $k\slash 3$' of the presentation complex for $\Gamma$. For random groups this graph $\Delta (\Gamma)$ can be written as the union of a graph $\Sigma_{2}$ and two bipartite graphs $\Sigma_{1},\Sigma_{3}. $ If we allowed all freely reduced words as relators, then these graphs would have the marginal distributions of Erd\"os--R\'enyi random graphs. Since we restrict to only having cyclically reduced words as relators, these graphs will not allow some edges, and so will have the marginal distributions of \emph{reduced random graphs}. We need to analyse the eigenvalues of these graphs, and then prove the union of these graphs has high eigenvalue with large probability.

The paper is structured as follows. Sections \ref{sec: graph defs} introduces some relevant graph theoretic definitions, and in Section \ref{sec: a spectral criterion for Property (T)} we provide a spectral criterion for Property (T), related to the graph $\Delta_{k}$. Sections \ref{sec: random graphs} and \ref{sec: Spectral theory of unions of reduced random graphs} are more geared towards graph theory, and allow us to analyse the eigenvalues of specific random graphs. In Section \ref{sec: Property (T) in random quotients of free groups} we apply these results prove the main theorems of this text.

We now briefly discuss some notation and assumptions. We are dealing with asymptotics, and so we frequently arrive at situations where $m$ is some parameter tending to infinity that is required to be an integer: if $m$ is not integer, we will implicitly replace it by $\lfloor m\rfloor$. Since we are dealing with asymptotics, this does not affect any of our arguments. 
\begin{definition}
Given $m_{1}:\mathbb{N}\rightarrow \mathbb{N}$ a function such that $m_{1}(n)\rightarrow\infty$ as $n\rightarrow\infty$, we write $m_{2}=m_{2}(m_{1}(n))$ to mean that $m_{2}(n)=f(m_{1}(n))$ for some function $f$, and $f(m_{1}(n))\rightarrow\infty$ as $n\rightarrow\infty$, i.e. $m_{2}$ only depends on $m_{1}$, and tends to infinity as $m_{1}$ tends to infinity.\end{definition}
The following are standard.
\begin{definition}
   Let $f,g:\mathbb{N}\rightarrow \mathbb{R}_{+}$ be two functions. We write 
   \begin{enumerate}
       \item $f=o(g)$ if $f(m)\slash g(m)\rightarrow 0$ as $m\rightarrow\infty$,
     
       \item $f=O(g)$ if there exists a constant $N\geq 0$ and $M\geq 1$ such that\\$f(m)\leq N g(m)$ for all $m\geq M$,
       \item and $f=\Omega (g)$ if $g=o(f)$.
   \end{enumerate}
   We write $f=o_{m}(g)$ etc to indicate that the variable name is $m$.
   
  Note that typically we will deal with functions $m_{2}=m_{2}(m_{1})$, and $f=f(m_{1},m_{2})$. We will write $f=o_{m_{1}}(g)$ etc to mean that the function $f'(m_{1})=f(m_{1},m_{2}(m_{1}))=o_{m_{1}}(g'(m_{1})),$ where $g'(m_{1})=g(m_{1},m_{2}(m_{1}))$.
\end{definition}

\begin{definition}
Let $\mathcal{M}(m)$ be some model of random groups (or graphs) depending on a parameter $m$, and let $\mathcal{P}$ be a property of groups (or graphs). We say that $\mathcal{P}$ holds \emph{asymptotically almost surely with} $m$ (\emph{a.a.s.}$(m)$) if 
$$\lim_{m\rightarrow\infty}\mathbb{P}(G\sim \mathcal{M}(m)\mbox{ has }\mathcal{P})=1.$$

Again, typically we will have deal with cases where $m_{2}=m_{2}(m_{1})$ is fixed, $\mathcal{M}(m_{1},m_{2})$ is some model of random groups (or graphs) depending on parameters $m_{1}$ and $m_{2}$, and $\mathcal{P}$ is a property of groups (or graphs). We say that $\mathcal{P}$ holds \emph{asymptotically almost surely with} $(m_{1})$ (\emph{a.a.s.}$(m_{1})$) if 
$$\lim_{m_{1}\rightarrow\infty }\mathbb{P}(G\sim \mathcal{M}(m_{1},m_{2}(m_{1}))\mbox{ has }\mathcal{P})=1.$$
Typically we will only use the above in proofs or in the statements of auxilliary technical lemmas.
\end{definition}
Finally, we will often be working with bipartite graphs: the vertex partition of a bipartite graph $G$ will always be written $V(G)=V_{1}(G)\sqcup V_{2}(G)$.
\section*{Funding}
Engineering and Physical Sciences Research Council Studentship 2114468.	
\section*{Acknowledgements}
I would like to thank my supervisor Henry Wilton for his support and guidance, for suggesting the study of Property (T) in random groups, and for many useful discussions. I would like to thank the anonymous referee for their detailed and helpful suggestions.

\section{Graphs and eigenvalues}\label{sec: graph defs}
In this short section we provide some definitions that will be central throughout.
\begin{definition}
A \emph{multiset} $M$ is a pair $M=(A,\mu_{M})$ where $A$ is a set and $\mu:A\rightarrow \mathbb{N}$ is a set function. We call $A$ the \emph{universe} of $M$ and $\mu_{M}$ its \emph{multiplicity}. Typically, in an abuse of notation, we write $M=(M,\mu_{M})$ to be a multiset, where we view $M$ as both the underlying universe, and the multiset.

Let $M=(A,\mu),N=(B,\nu)$ be multisets. The \emph{sum} of $M$ and $N$ is the multiset defined by $$M\uplus N:=(A\cup B,\mu+\nu),$$
where we extend $\mu\vert_{B\setminus A}:=0,\; \nu\vert_{A\setminus B}:=0.$
\end{definition}
\begin{definition}
A graph is a pair $G=(V,E)$, where $V$ is the \emph{set of vertices}, and $E=(E,\mu_{E})$ is a multiset. We typically refer to $E$ as the \emph{set of edges}, which consists of unordered pairs of the form $\{u,v\}$, for $u,v\in V$. Note that here we allow pairs $\{u,u\}$.
An edge $\{u,v\}$ is said to \emph{join} the vertices $u$ and $v$. We refer to $\mu_{E}(\{u,v\})$ as the \emph{number of edges} joining $u$ and $v$.
\end{definition}

\begin{definition}
Let $G=(V,E)$, $G'=(V',E')$. 
The \emph{union} of $G$ and $G'$ is the graph $G\cup G':=(V\cup V',E\uplus E')$.
\end{definition}

Let $G=(V,E)$ be a graph with vertex set $V=\{v_{1},\hdots ,v_{m}\}$. The \emph{adjacency matrix} of $G$, $A(G)$, is the $m\times m$ matrix with $A(G)_{i,j}$ defined to be the number of edges between $v_{i}$ and $v_{j}$, i.e. $A(G)_{i,j}=\mu_{E}(\{v_{i},v_{j}\})$. The \emph{degree matrix} of $G$, $D(G)$, is the diagonal matrix with entries $D(G)_{i,i}=deg(v_{i}):=\sum_{v_{j}}\mu_{E}(\{v_{i},v_{j}\})$. The \emph{Laplacian} of $G$, $L(G)$, is defined by $$L(G)=I-D^{-1\slash 2}AD^{-1\slash 2}.$$
We note that $L(G)$ is symmetric positive semi-definite, with eigenvalues $$0\leq \lambda_{0}(L(G))\leq \lambda_{1}(L(G))\leq \hdots \leq \lambda_{m-1}(L(G))\leq 2.$$ For $i=1,\hdots, m$, we define $\lambda_{i}(G):=\lambda_{i}(L(G))$.

We note the following lemma, commonly known as Weyl's inequality, which will also be of frequent use. If $A$ is a symmetric real $m\times m$ matrix, then $A$ has real eigenvalues, which we order by $\lambda_{0}(A)\leq \lambda_{1}(A)\leq \hdots \leq \lambda_{m-1}(A)$. We define the reverse ordering of eigenvalues $\mu_{1}(A)\geq \mu_{2}(A)\geq \hdots\geq \mu_{m}(A)$, i.e. $\mu_{i}(A)=\lambda_{m-i}(A)$.

\begin{lemma*}[Weyl's inequality \cite{Weyl}]
Let $A$ and $B$ be symmetric $m\times m$ real matrices. For $i=1,\hdots ,m$: $\mu_{i}(A)+\mu_{m}(B)\leq \mu_{i}(A+B)\leq \mu_{i}(A)+\mu_{1}(B).$
\end{lemma*}
We also make the following remarks.

\begin{remark}
Let $M$ be a symmetric $n\times n$ matrix. For $i=1,\hdots , n$: $$\mu_{i}(-M)=-\mu_{n+1-i}(M).$$
This follows as $\{\mu_{i}(-M)\;:\;1\leq i\leq n\}=\{-\mu_{i}(M)\;:\;1\leq i\leq n\},$
and $\mu_{1}(M)\geq \mu_{2}(M)\geq \hdots\geq \mu_{n}(M),$ so that $-\mu_{1}(M)\leq -\mu_{2}(M)\leq \hdots\leq -\mu_{m}(M).$
\end{remark}
\begin{remark}\label{rmk: switching between eigenvalues of graphs}
Let $G$ be a graph. For $i=0,\hdots , \vert V(G)\vert -1$:
$$\lambda_{i}(G)=1-\mu_{i+1}(D(G)^{-1\slash 2}A(G)D(G)^{-1\slash 2}),$$
since $L(G)=I-D(G)^{-1\slash 2}A(G)D(G)^{-1\slash 2}$, so that 
$$\{\lambda_{i}(L(G))\;:\;0\leq i\leq \vert V(G)\vert -1\}=\{1-\mu_{j}(D(G)^{-1\slash 2}A(G)D(G)^{-1\slash 2})\;:\;1\leq j\leq \vert V(G)\vert\},$$
and
\begin{equation*}
    \begin{split}
        1-\mu_{1}(D(G)^{-1\slash 2}A(G)D(G)^{-1\slash 2})&\leq 1-\mu_{2}(D(G)^{-1\slash 2}A(G)D(G)^{-1\slash 2})\leq \hdots\\
        \hdots &\leq 1-\mu_{\vert V(G)\vert}(D(G)^{-1\slash 2}A(G)D(G)^{-1\slash 2}).
    \end{split}
\end{equation*}
\end{remark}


\section{A spectral criterion for Property (T)}\label{sec: a spectral criterion for Property (T)}
In this section we deduce a spectral criterion for Property (T): we first remind the reader of some of the relevant definitions. We focus only on finitely generated discrete groups: for a further exposition the reader should see, for example, \cite{bekkadelaharpe}.

Let $\Gamma$ be a finitely generated group with finite generating set $S$, let $\mathcal{H}$ be a Hilbert space, and let $\pi: \Gamma\rightarrow\mathcal{U}(\mathcal{H})$ be a unitary representation of $\Gamma.$ We say that $\pi$ has \emph{almost-invariant vectors} if for every $\epsilon>0$ there is some non-zero $u_{\epsilon}\in\mathcal{H}$ such that for every $s\in S$, $||\pi(s) u_{\epsilon}-u_{\epsilon}||<\epsilon ||u_{\epsilon}||.$
\begin{definition}
We say that $\Gamma$ has \emph{Property (T)} if for every Hilbert space $\mathcal{H}$, and for every unitary representation $\pi:\Gamma\rightarrow\mathcal{U}(\mathcal{H})$ with almost-invariant vectors, there exists a non-zero invariant vector for $\pi$.
\end{definition}
It is standard that the choice of generating set does not matter. We now note the following well known results concerning Property (T): for proofs see, for example, \cite{bekkadelaharpe}. We will use these results implicitly throughout.

\begin{lemma*}
Let $\Gamma$ be a finitely generated group, and let $H$ be a finite index subgroup of $\Gamma$: $\Gamma$ has Property (T) if and only if $H$ has Property (T).
\end{lemma*}
\begin{lemma*}
Let $\Gamma$ be a finitely generated group with Property (T) and let $\Gamma'$ be a homomorphic image of $\Gamma.$ Then $\Gamma'$ has Property (T).
\end{lemma*}

\subsection{A spectral criterion for Property (T)}\label{subsec: spectral criterion}

Now, let $\Gamma=\pres{A_{n}}{R}$ be a finite presentation of a group. Let $R_{k}$ be the set of words in $R$ of length $k$. Define the graph $\Del{3}$ by $$V(\Del{3})=A_{n}\sqcup A_{n}^{-1}$$ and for each relator $r=r_{1}r_{2}r_{3}\in R_{3}$ add the edges $$(r_{1},r_{3}^{-1}),(r_{2},r_{1}^{-1}),(r_{3},r_{2}^{-1}).$$

The use of this graph is the following, proved independently by \cite{zuk1996} (c.f. \cite{Zuk}) and \cite{Ballmann-Swiatkowski}. The result is often stated for a model of $\Delta_{3}$ without multiple edges, and is often known as \emph{\.{Z}uk's criterion for Property (T)}.
\begin{theorem}\label{thm: Zuk's criterion}\cite{zuk1996,Ballmann-Swiatkowski}
Let $\Gamma=\pres{A_{n}}{R}$ be a finite presentation. If $\lambda_{1}(\Del{3})>1\slash 2$, then $\Gamma$ has Property (T).
\end{theorem}

We now apply this to recover an alternate spectral criterion for Property (T). However, before we introduce the graph $\Delta_{k}$, we first note a result regarding finite index subgroups of free groups. For the free group $F_{n}:=\mathbb{F}(A_{n})$ and $l\geq 1$, we define $\W{l}$ to be the set of freely reduced words of length $l$ in $F_{n}$. We now prove that these sets always generate finite index subgroups of $F_{n}$.

\begin{lemma}\label{lem: finite index subgroups of free groups}
Let $l\geq 1$. Then $[F_{n}:\langle \W{l}\rangle]<\infty$.

\end{lemma}
(In fact it is easily seen that $[F_{n}:\langle \W{l}\rangle]\leq 2$).
\begin{proof}
Note that $\W{l}=S_{l}(F_{n})$, the sphere of radius $l$ in $F_{n}$. Hence $$[F_{n}:\langle \W{l}\rangle]\leq \vert B_{F_{n}}(id,l-1)\vert=2n(2n-1)^{l-2},$$ since $F_{n}=B_{F_{n}}(id,l-1)\langle S_{l}(F_{n})\rangle.$
\end{proof}

We now introduce the graph to which our spectral criterion will apply.
\begin{definition}
Let $G=\pres{A_{n}}{R}$ be a finite presentation of a group and let $k\geq 3$. We define the graph $\Delta_{k}( A_{n}\;\vert\;R)$, as follows, depending on $k\mod 3$.

\begin{enumerate}[label=$\bullet\;\mathbf{k=\arabic*\mbox{\bf\; mod }3:}$,start=0]
    \item 
 Let $V(\Del{k})=\W{k\slash 3}$. For each relator\\ $r=r_{1}\hdots r_{k}\in R_{k}$, write $r=r_{x}r_{y}r_{z}$ with $r_{x},r_{y},r_{z}\in \W{k\slash 3}$, and add the edges 
$$(r_{x},r_{z}^{-1}),\; (r_{y},r_{x}^{-1}),\;(r_{z},r_{y}^{-1}).$$

    \item
  Let $\Del{k}$ be the graph with $$V(\Del{k})=\bW{\frac{k-1}{3}}\bigsqcup\bW{\frac{k+2}{3}}.$$ For each relator $r=r_{1}\hdots r_{k}\in R_{k}$ write $r=r_{x}r_{y}r_{z}$ with $r_{x},r_{y}\in \W{\frac{k-1}{3}}$ and $r_{z}\in \W{\frac{k+2}{3}}$, and add the edges 
$$(r_{x},r_{z}^{-1}),\; (r_{y},r_{x}^{-1}),\;(r_{z},r_{y}^{-1}).$$

    \item 
 Let $\Del{k}$ be the graph with $$V(\Del{k})=\bW{\frac{k-2}{3}}\bigsqcup\bW{\frac{k+1}{3}}.$$ For each relator $r=r_{1}\hdots r_{k}\in R_{k}$ write $r=r_{x}r_{y}r_{z}$ with $r_{x},r_{y}\in \W{\frac{k+1}{3}}$ and $r_{z}\in \W{\frac{k-2}{3}}$, and add the edges 
$$(r_{x},r_{z}^{-1}),\; (r_{y},r_{x}^{-1}),\;(r_{z},r_{y}^{-1}).$$

\end{enumerate}

\end{definition}
We can prove the following.

\begin{lemma}\label{lem: spectral criterion for Property (T)}
Let $\Gamma=\pres{A_{n}}{R}$ be a finite presentation and let $k\geq 3$. If $\lambda_{1}(\Del{k})>1\slash 2$, then $\Gamma$ has Property (T).
\end{lemma}
We note that this Lemma is not particularly effective when given a specific finite presentation of a group: for the above spectral condition to hold, we heuristically require $\vert R\vert>>(2n-1)^{(k+(-k\mod 3))\slash 3}$. 
However, this is exactly the regime we consider for random groups.
\begin{proof}
We prove this for $k=2\mod 3$: the other cases are similar. First, for ease, let $\Gamma'=\pres{A_{n}}{R_{k}}$. Since $\Gamma$ is a homomorphic image of $\Gamma'$, it suffices to prove that $\Gamma'$ has Property (T). 
Let $\phi: F_{n}\twoheadrightarrow \Gamma '$ be the canonical epimorphism induced by the choice of presentation for $\Gamma '$. Let $\mathcal{W}=\W{(k-2)\slash 3}\sqcup \W{(k+1)\slash 3},$ $W=\phi(\mathcal{W}),$ and let $H=\langle W \rangle_{\Gamma '}$: by Lemma \ref{lem: finite index subgroups of free groups} we have that $[\Gamma' : H]<\infty $.

For each $r\in R_{k}$, write $r=r_{x}r_{y}r_{z}$ where $r_{x},r_{y}\in \W{(k+1)\slash 3}$ and $r_{z}\in\W{(k-2)\slash3}$. Let $T=\{r_{x}r_{y}r_{z}\;:\;r\in R_{k}\}$ and let $$\tilde{\Gamma}:=\mathbb{F}(\mathcal{W})\bigg\slash \langle \langle T\rangle\rangle=\langle \mathcal{W}\;\vert\;T\rangle.$$

It is clear that there is a surjective homomorphism $\psi:\tilde{\Gamma}\twoheadrightarrow H$, so that $\Gamma'$ is virtually a homomorphic image of $\tilde{\Gamma}$. Next, we note that $\Delta_{k}( A_{n}\;\vert\;R)\cong \Delta_{3}( \mathcal{W}\;\vert\;T).$ By Theorem \ref{thm: Zuk's criterion}, if $\lambda_{1}(\Del{k})=\lambda_{1}(\Delta_{3}( \mathcal{W}\;\vert\;T))>1\slash 2,$ then $\tilde{\Gamma}$ has Property (T). Since Property (T) is preserved under epimorphisms and passing to finite index extensions, it follows that if $\lambda_{1}(\Del{k})>1\slash 2$, then $\Gamma$ has Property (T).
\end{proof}


\section{The spectral theory of almost regular graphs, Erd\"os--R\'enyi random graphs, and the unions of regular graphs}\label{sec: random graphs}
In this section we analyse the spectral theory of almost regular graphs, as well as some results on the eigenvalues of Erd\"os--R\'enyi random graphs. We also prove a result concerning the eigenvalues of the union of a well connected graph and two bipartite graphs. We first note the following lemma.

\begin{lemma}\label{lem: max eigenvalue of adjacency matrix}
Let $G$ be a graph. Then 
$\max_{i}\vert \mu_{i}(A(G))\vert\leq \max_{v\in V(G)}deg(v).$ If $G$ is bipartite, then 
$$\max_{i}\vert \mu_{i}(A(G))\vert\leq \max_{\substack{v\in V_{1}(G)\\w\in V_{2}(G)}}\sqrt{deg(v)deg(w)}.$$
\end{lemma}
\begin{proof}
The first result follows as $||A(G)||_{\infty}=\max_{v\in V(G)}deg(v)$, and it is standard that $||A(G)||_{\infty}$ is an upper bound for the absolute values of the eigenvalues of $A(G)$. 

The second inequality follows from e.g. \cite[3.7.2]{hornjohnson}, as follows. In this case, we have $$A(G)=\begin{pmatrix}
0&B\\B^{T}&0
\end{pmatrix},$$ for some matrix $B$. By definition, the set of eigenvalues of $A$ are the set of \emph{singular values} of $B$, $\{\sigma_{j}(B)\}_{j}$. Therefore, $\max_{i}\vert \lambda_{i}(A(G))\vert =\max_{i}\vert \sigma_{i}(B)\vert$. By \cite[3.7.2]{hornjohnson}, $$\max_{i}\vert \sigma_{i}(B)\vert\leq \sqrt{\vert \vert B\vert \vert_{\infty} \vert \vert B\vert \vert_{1}}=\max_{\substack{v\in V_{1}(G)\\w\in V_{2}(G)}}\sqrt{deg(v)deg(w)}.$$
\end{proof}


\subsection{The spectra of almost regular graphs}

We now analyse the spectra of almost regular graphs. These definitions are standard in graph theory and appear in e.g. \cite{kotowski}.
\begin{definition}[Almost regular graphs]
Let $\{G_{m}\}_{m=1}^{\infty}$ be a collection of graphs. We say that the graphs $G_{m}$ are \emph{almost $d_{m}$-regular} if for every $G_{m}$ its minimum and maximum degree are $(1+o_{m}(1))d_{m}$.
\end{definition}
\begin{definition}[Almost regular bipartite graphs]
Let $\{G_{m}\}_{m=1}^{\infty}$ be a collection of bipartite graphs. We say that the graphs $G_{m}$ are \emph{almost $(d^{(1)}_{m},d^{(2)}_{m})$-regular} if for every $G_{m}$ the minimum and maximum degree of vertices in $V_{1}(G_{m})$ are $(1+o_{m}(1))d^{(1)}_{m}$ and the minimum and maximum degree of vertices in $V_{2}(G_{m})$ are $(1+o_{m}(1))d^{(2)}_{m}$.
\end{definition}

We note the following results.
\begin{lemma*}\cite[Lemma 4.4]{kotowski}
Let $d_{m}\rightarrow \infty $ and let $G_{m}$ be almost $d_{m}$-regular. Then $\frac{1}{d_{m}}\mu_{2}(A(G_{m}))=(1+o_{m}(1))(1-\lambda_{1}(G_{m}))$. In particular, if $\mu_{2}(A(G_{m}))=o_{m}(d_{m})$ then $\lambda_{1}(G_{m})=1-o_{m}(1)$.
\end{lemma*}
\begin{lemma*}\cite[Lemma 4.5]{kotowski}
Let $G_{m}$ be an almost $d_{m}$-regular graph and let $G'_{m}$ be a graph on the same vertex set whose maximum degree is $o_{m}(d_{m})$. Then:
\begin{enumerate}[label=$\roman*)$]
    \item $G_{m}\cup G'_{m}$ is almost $d_{m}$ regular,
    \item and $\lambda_{1}(G_{m})=\lambda_{1}(G_{m}\cup G'_{m})+o_{m}(1).$
\end{enumerate}
\end{lemma*}
Again, recall that $\lambda_{1}(G)=1-\mu_{2}(D^{-1\slash 2}AD^{-1\slash 2})$.
We now prove the corresponding result for bipartite graphs: our proofs are different to \cite{kotowski}, and rely on Weyl's inequality.

\begin{lemma}\label{lem: bipartite almost regular spectral calculation}
Let $d^{(1)}_{m},d^{(2)}_{m}\rightarrow \infty $ and let $G_{m}$ be almost $(d^{(1)}_{m},d^{(2)}_{m})$-regular. For $i=1,\hdots , \vert V(G_{m})\vert $: $$\frac{1}{\sqrt{d^{(1)}_{m}d^{(2)}_{m}}}\mu_{i}(A(G_{m}))=\mu_{i}(D^{-1\slash 2}(G_{m})A(G_{m})D^{-1\slash 2}(G_{m}))+o_{m}(1).$$  In particular, if $\mu_{2}(A(G_{m}))=o_{m}\bigg(\sqrt{d^{(1)}_{m}d^{(2)}_{m}}\bigg),$ then $\lambda_{1}(G_{m})=1-o_{m}(1)$.
\end{lemma}
\begin{proof}
  
As $G_{m}$ is almost $(d^{(1)}_{m},d^{(2)}_{m})$-regular, we see that for $$A=A(G_{m})=\begin{pmatrix}
0&A_{1}\\
A_{1}^{T}&0
\end{pmatrix},\;D=D(G_{m}),$$
there exists a matrix $K$ with norm $\vert\vert K\vert\vert_{\infty}=o_{m}(1)$ such that $$\frac{1}{\sqrt{d_{m}^{(1)}d_{m}^{(2)}}}A=D^{-1\slash 2}AD^{-1\slash 2}+K.$$

Since $\vert \mu_{i}(K)\vert \leq \vert \vert K\vert \vert_{\infty}=o_{m}(1)$ for all $i$, the first statement of the Lemma follows easily by Weyl's inequality. The second statement follows from Remark \ref{rmk: switching between eigenvalues of graphs}.
\end{proof}
\begin{lemma}\label{lem: spectra under addition of small bipartite graph}
Let $d^{(1)}_{m},d^{(2)}_{m}\rightarrow \infty $, and let $G_{m}$ be almost $(d^{(1)}_{m},d^{(2)}_{m})$-regular. Let $G'_{m}$ be a bipartite graph on the same vertex set as $G_{m}$ with the same vertex partitions, such that the maximum degree of $v\in V_{i}(G'_{m})$ is $o_{m}(d^{(i)}_{m})$. Then:
\begin{enumerate}[label=$\roman*)$]
    \item $G_{m}\cup G'_{m}$ is almost $(d^{(1)}_{m},d^{(2)}_{m})$ regular,
    \item and $\lambda_{1}(G_{m})=\lambda_{1}(G_{m}\cup G'_{m})+o_{m}(1).$
\end{enumerate}
\end{lemma}
\begin{proof}
Part $i)$ is immediate. For part $ii)$, we see that $A(G_{m}\cup G'_{m})=A(G_{m})+A(G'_{m})$: since the maximum degree of a vertex $v\in V_{i}(G'_{m})=V_{i}(G_{m})$ is $o(d_{m}^{(i)})$, we have by Lemma \ref{lem: max eigenvalue of adjacency matrix} that $\max_{i}\vert \mu_{i}(A(G'_{m}))\vert \leq o_{m}\bigg(\sqrt{d_{m}^{(1)}d_{m}^{(2)}}\bigg)$, and hence $$\max_{i}\bigg\vert \mu_{i}\bigg(\frac{1}{\sqrt{d_{m}^{(1)}d_{m}^{(2)}}}A(G'_{m})\bigg)\bigg\vert =o_{m}(1).$$ By Weyl's inequality, 
\begin{equation*}
    \begin{split}
        &\mu_{2}\bigg(\frac{1}{\sqrt{d_{m}^{(1)}d_{m}^{(2)}}}(A(G_{m})+A(G'_{m}))\bigg)\\
        &\leq \mu_{2}\bigg(\frac{1}{\sqrt{d_{m}^{(1)}d_{m}^{(2)}}}A(G_{m})\bigg)+\mu_{1}\bigg(\frac{1}{\sqrt{d_{m}^{(1)}d_{m}^{(2)}}}A(G'_{m})\bigg) \\
        &=\mu_{2}\bigg(\frac{1}{\sqrt{d_{m}^{(1)}d_{m}^{(2)}}}A(G_{m})\bigg)+o_{m}(1).
    \end{split}
\end{equation*}
Similarly $$\mu_{2}\bigg(\frac{1}{\sqrt{d_{m}^{(1)}d_{m}^{(2)}}}(A(G_{m})+A(G'_{m}))\bigg)\geq \mu_{2}\bigg(\frac{1}{\sqrt{d_{m}^{(1)}d_{m}^{(2)}}}A(G_{m})\bigg)+o_{m}(1),$$
and the result follows by Remark \ref{rmk: switching between eigenvalues of graphs} and Lemma \ref{lem: bipartite almost regular spectral calculation}.
\end{proof}


\subsection{Almost regularity of Erd\"os--R\'enyi random graphs and their eigenvalues}
In this section, we introduce some models of random graphs, and then prove they are almost regular.

\begin{definition}[\emph{Erd\"os--R\'enyi random graph}]
Let $m\geq 1$ and $0<p:=p(m)<1$. The Erd\"os--R\'enyi random graph $G(m,p)$ is the random graph model with vertex set $\{u_{1},\hdots ,u_{m}\}$ and edge set obtained by adding each edge $\{u_{i},u_{j}\}$ independently with probability $p$. For a random graph $G$ we write $G\sim G(m,p)$ to indicate that the distribution of $G$ is that of $G(m,p)$.
\end{definition}

\begin{definition}[\emph{Erd\"os--R\'enyi random bipartite graph}]
Let $m_{1},m_{2}\geq 1$ and let $0<p:=p(m_{1},m_{2})<1$. The Erd\"os--R\'enyi random bipartite graph $G(m_{1},m_{2},p)$ is the random bipartite graph model with vertex set $V_{1}=\{u_{1},\hdots ,u_{m_{1}}\},\; V_{2}=\{v_{1},\hdots ,v_{m_{2}}\},$ and edge set obtained by adding each edge $\{u_{i},v_{j}\}$ independently with probability $p$.
\end{definition}
Given a model of random graphs $\mathcal{M}$, and a random matrix $M$, we write $M\sim A(\mathcal{M})$ to indicate that the distribution of $M$ is the same as that obtained by sampling a graph $G\sim \mathcal{M}$ and then taking its adjacency matrix.

We now analyse the regularity of random bipartite graphs. For this we will use the Chernoff bounds: for $X\sim Bin(n,p)$ and $\delta\in [0,1],$
$$\mathbb{P}(\vert X - np\vert \geq \delta np)\leq 2exp(-np \delta^{2}\slash 3).$$
\begin{lemma}\label{lem: almost regular G(m,n,p)}
Let $m_{2}=m_{2}(m_{1})$ and $p=p(m_{1},m_{2})=p(m_{1})$ be such that $\min\{m_{1},m_{2}\}p=\Omega _{m_{1}} (\log \max\{m_{1},m_{2}\})$. Then \aas{m_{1}} $G(m_{1},m_{2},p)$ is almost $(m_{2}p,m_{1}p)$-regular.
\end{lemma}
\begin{proof}
First note $m_{2}p\geq  \omega \log m_{1}$ and $m_{1}p\geq \omega \log m_{2}$ for some $\omega\rightarrow\infty$ as $m_{1}\rightarrow\infty$. Let $G\sim G(m_{1},m_{2},p)$. Let $v\in V_{1}(G)$, $w\in V_{2}(G)$. Note that $\mathbb{E}(deg(v))=m_{2}p$, $\mathbb{E}(deg(w))=m_{1}p,$ and $Var(deg(v))=m_{2}p(1-p),$ $Var(deg(w))=m_{1}p(1-p).$ Let $\epsilon=\omega^{-1\slash 3}$. By the Chernoff bounds, for a fixed vertex $v$ in $V_{1}$: 
$$\mathbb{P}(\vert deg(v)-m_{2}p\vert\geq \epsilon m_{2}p)\leq 2 exp\{-\epsilon^{2}m_{2}p\slash 3\}.$$
Hence the probability that there exists a vertex in $V_{1}$ with degree too large or small is:
\begin{equation*}
\begin{split}
\mathbb{P}(\exists\; v\in V_{1}\;:\;\vert deg(v)-m_{2}p\vert\geq \epsilon m_{2}p)&\leq 2m_{1}exp\{-\epsilon^{2}m_{2}p\slash 3\}\\
&\leq 2m_{1}exp\{-\omega^{1\slash 3}\log m_{1}\slash 3\}\\
&=2m_{1}^{-\Omega_{m_{1}} (1)}.  
\end{split}
\end{equation*}

Similarly: 
\begin{equation*}
\begin{split}
\mathbb{P}(\exists \;w\in V_{2}\;:\;\vert deg(w)-m_{1}p\vert\geq \epsilon m_{1}p)&\leq 2m_{2}exp\{-\epsilon^{2}m_{1}p\slash 3\}\\
&\leq2 m_{2}exp\{-\omega^{1\slash 3}\log m_{2}\slash 3\}\\
&=2m_{2}^{-\Omega_{m_{1}}(1)}.
\end{split}
\end{equation*}
\end{proof}
Therefore we immediately see the following.

\begin{lemma}\label{lem: leading eigenvalue of random bipartite graph}
Let $m_{2}=m_{2}(m_{1})$ and $p=p(m_{1})$ be such that $\min\{m_{1},m_{2}\}p=\Omega_{m_{1}}( \log \max\{m_{1},m_{2}\})$. Then a.a.s.$(m_{1})$ $$\mu_{1}(A(G(m_{1},m_{2},p)))\leq [1+o_{m_{1}}(1)]p\sqrt{m_{1}m_{2}}.$$
\end{lemma}
\begin{proof}
By Lemma \ref{lem: almost regular G(m,n,p)}, a.a.s.$(m_{1})$ the maximum degree of a vertex in $V_{1}$ is $(1+o_{m_{1}}(1))m_{2}p,$ and the maximum degree of a vertex in $V_{2}$ is $(1+o_{m_{1}}(1))m_{1}p$. By Lemma \ref{lem: max eigenvalue of adjacency matrix}, 
$$\max_{i}\vert \mu_{i}(A(G(m_{1},m_{2}p)))\vert \leq \max_{\substack{v\in V_{1}(G)\\w\in V_{2}(G)}}\sqrt{deg(v)deg(w)}\leq [1+o_{m_{1}}(1)]\sqrt{m_{1}m_{2}p^{2}}$$
with probability tending to $1$ as $m_{1}$ tends to infinity.

\end{proof}

Similarly we can deduce that the Erd\"os--R\'enyi random graph is almost regular. \begin{lemma}\label{lem: almost regular G(m,p)}
Let $m\geq 1$ and $p=p(m)$ be such that $mp=\Omega _{m} (\log m)$. Then \aas{m} $G(m,p)$ is almost $mp$-regular.
\end{lemma}

We now note some results on the eigenvalues of Erd\"os--R\'enyi random graphs. The eigenvalues of $G(m,p)$ were first analysed by \cite{furedikomlos}: we use the following result, due to \cite{furedikomlos} (an extension to a more general model can be found in \cite{chungrandomgraph}).
\begin{theorem}\label{lem: eigenvalue of ER random graph} \cite[Theorem 1]{furedikomlos}
Let $p>0$ be such that $mp=\Omega_{m}(\log^{6} (m))$, and let $G\sim G(m,p)$. Then a.a.s.$(m)$, $$\max_{i\neq 1}\vert \mu_{i}(A(G))\vert \leq 2[1+o_{m}(1)]\sqrt{mp},$$ $$\max_{i\neq 0}\vert 1-\lambda_{i}(G)
\vert =o_{m}(1).$$
\end{theorem}
The eigenvalues of the bipartite version, $G(m_{1},m_{2},p)$ were analysed far more recently: see, e.g., \cite[Theorem A]{ashcroft2021eigenvalues}.
\begin{theorem}\label{lem: eigenvalue of ER bipartite random graph}
Let $m_{1}\geq 1$, $m_{2}=m_{2}(m_{1})$, and $p=p(m_{1})$ be such that: $m_{1}p=\Omega(\log^{6}m_{1}),$ $m_{2}p=\Omega(\log^{6}m_{2}).$ Let $G\sim G(m_{1},m_{2},p).$ Then with probability tending to $1$ as $m_{1}$ tends to infinity: 
$$\max\limits_{i\neq 0,m_{1}+m_{2}-1}\bigg\vert1-\lambda_{i}(G)\bigg\vert=o_{m_{1}}(1).$$
\end{theorem}

\subsection{Spectra of unions of regular graphs}
The purpose of this subsection is to analyse the spectral distribution of unions of three graphs with relatively high first eigenvalue. This is already known when all three graphs share the same vertex set.
\begin{lemma}\label{lem: eigenvalue of union of three graphs on same vertex set}\cite[p. 665]{Zuk}
Let $G_{1},G_{2},G_{3}$ be $d$-regular graphs on the same vertex set, and suppose $\lambda_{1}(G_{i})>1-c$ for each $i$. Then $$\lambda_{1}(G_{1}\cup G_{2}\cup G_{3})\geq 1-c.$$
\end{lemma}
We now wish to extend this to the case where the graphs are relatively well connected, and they do not share the same vertex set. We first recall (a partial consequence of) the Courant-Fischer Theorem, as follows.

\begin{theorem*}[Courant-Fischer Theorem]
Let $M$ be a symmetric $m\times m$ matrix with first eigenvalue $\mu_{1}(M)$ and corresponding eigenvector $\underline{e}$. Then 
$$\mu_{2} (M)= \max\limits_{\substack{\vect{x}\perp \underline{e}\\ \vert \vert \vect{x}\vert \vert =1}}\langle M\vect{x},\vect{x}\rangle=\max\limits_{\vect{x}\perp \underline{e}}\dfrac{\langle M\vect{x},\vect{x}\rangle}{\langle \vect{x},\vect{x}\rangle}.$$
\end{theorem*}
Hence we can prove the following (recall that for a bipartite graph $G$, $V_{1}(G)$ and $V_{2}(G)$ are the vertex partitions of $G$).
\begin{lemma}\label{thm: eigenvalue of union of three graphs}
Let $G_{1},G_{2},G_{3}$ be graphs such that:
\begin{enumerate}[label=$\roman *)$]
    \item $G_{2},G_{3}$ are bipartite, $V(G_{1})=V_{1}(G_{2})=V_{1}(G_{3})$, and $V_{2}(G_{2})=V_{2}(G_{3}),$
    \item $G_{1}$ is $2d_{1}$-regular, and $G_{2},G_{3}$ are $(d_{1},d_{2})$-regular,
    \item and for $i=1,2,3$ there exists $0\leq c_{i}<1$ with $\lambda_{1}(G_{i})\geq 1-c_{i}$.
\end{enumerate}
Then $$\lambda_{1}(G_{1}\cup G_{2}\cup G_{3})\geq 1-\dfrac{\sqrt{2}c_{1}+c_{2}+c_{3}}{2\sqrt{2}}.$$
\end{lemma}
\begin{proof}
Let $\vect{1}_{l}$ be the all $1$ vector with $l$ entries, and let $G=G_{1}\cup G_{2}\cup G_{3}$. Let $V_{1}=V(G_{1})= V_{1}(G_{2})=V_{1}(G_{3})$ and $V_{2}= V_{2}(G_{2})=V_{2}(G_{3})$. Let $m_{1}=\vert V_{1}\vert$ and $m_{2}=\vert V_{2}\vert$. For $i=1,2,3$, let $\Lambda_{i}=D_{i}^{-1\slash 2}A_{i}D_{i}^{-1\slash 2}$ where $D_{i}=D(G_{i}),A_{i}=A(G_{i})$ (here we view $G_{1}$ as a graph on $V_{1}\sqcup V_{2}$).  Let $D=D(G)$, $A=A(G)$, and consider $\Lambda = D(G)^{-1\slash 2}A(G)D(G)^{-1\slash 2}$, so that $$\Lambda = \frac{1}{2}\Lambda_{1}+\frac{1}{2\sqrt{2}} \Lambda_{2}+\frac{1}{2\sqrt 2}\Lambda_{3}:$$
each of $\Lambda, \Lambda_{1},\Lambda_{2},\Lambda_{3}$ is symmetric and hence self-adjoint.
We remark again that $\mu_{2}(\Lambda_{i})=1-\lambda_{1}(G_{i}).$ We also note that $m_{2}d_{2}=m_{1}d_{1},$ so that $d_{2}=m_{1}d_{1}\slash m_{2}$.

Now, we consider the first eigenvalues of the matrices $\Lambda$ and $\Lambda_{i}$. 
The eigenvector corresponding to $\mu_{1}(\Lambda)=1$ is $$D^{1\slash 2}\vect{1}_{m_{1}+m_{2}}=\begin{pmatrix}
2\sqrt{d_{1}}\vect{1}_{m_{1}}\\
\sqrt{2d_{2}}\vect{1}_{m_{2}}
\end{pmatrix}.$$ The eigenvector corresponding to $\mu_{1}(\Lambda_{2})=1$ and $\mu_{1}(\Lambda_{3})=1$ is $$D_{2}^{1\slash 2}\vect{1}_{m_{1}+m_{2}}=D_{3}^{1\slash 2}\vect{1}_{m_{1}+m_{2}}=\begin{pmatrix}
\sqrt{d_{1}}\vect{1}_{m_{1}}\\
\sqrt{d_{2}}\vect{1}_{m_{2}}
\end{pmatrix}.$$ The eigenvector corresponding to $\mu_{1}(\Lambda_{1})=1$ is $$D_{1}^{1\slash 2}\vect{1}_{m_{1}+m_{2}}=\begin{pmatrix}
\sqrt{2d_{1}}\vect{1}_{m_{1}}\\
0\end{pmatrix}.$$ 

Let $\vect{\phi}$ be a vector with $\norm{\phi}=1,$ $\vect{\phi}\cdot D^{1\slash 2}\vect{1}_{m_{1}+m_{2}}=0$, and $\mu_{2} (\Lambda)= \langle \Lambda\vect{\phi},\vect{\phi}\rangle,$ which exists by the Courant-Fischer Theorem. We may write $$\vect{\phi}=\begin{pmatrix}
\alpha \vect{1}_{m_{1}}+\vect{u}\\
\beta\vect{1}_{m_{2}}+\vect{v}
\end{pmatrix}$$
where $\vect{u}\cdot \vect{1}_{m_{1}}=\vect{v}\cdot\vect{1}_{m_{2}}=0.$ As $$\vect{\phi}\cdot D^{1\slash 2}\vect{1}_{m_{1}+m_{2}}=2\sqrt{d_{1}}\alpha m_{1}+\sqrt{2d_{2}}\beta m_{2}=2\sqrt{d_{1}}\alpha m_{1}+\sqrt{2d_{1}m_{1}\slash m_{2}}\beta m_{2},$$
we see $$\beta = \dfrac{-\sqrt{2m_{1}}\alpha }{\sqrt{m_{2}}}.$$ Let 
$$\vect{\phi_{1}}=\begin{pmatrix}
\alpha \vect{1}_{m_{1}}\\
\beta \vect{1}_{m_{2}}
\end{pmatrix},\;\vect{\phi_{2}}=\begin{pmatrix}
\vect{u}\\
\vect{v}
\end{pmatrix},$$
so that $\vect{\phi_{1}}\cdot D^{1\slash 2}\vect{1}_{m_{1}+m_{2}}=\vect{\phi_{2}}\cdot D^{1\slash 2}\vect{1}_{m_{1}+m_{2}}=0.$
Write $\gamma = \norm{\phi_{1}}^{2},$ with $\norm{\phi_{2}}^{2}=1-\gamma.$
Note that 
\begin{equation*}
    \gamma=\alpha^{2}m_{1}+\beta^{2}m_{2}=3\alpha^{2}m_{1},
\end{equation*}
so that $3\alpha^{2}m_{1}\leq 1.$
We now calculate:
\begin{equation*}
    \begin{split}
        \langle \Lambda_{1}\vect{\phi_{1}},\vect{\phi_{1}}\rangle =\begin{pmatrix}
\alpha \vect{1}_{m_{1}}\\
0
\end{pmatrix}\cdot  \begin{pmatrix}
\alpha \vect{1}_{m_{1}}\\
\beta \vect{1}_{m_{2}}
\end{pmatrix}
=\alpha^{2}m_{1}.
    \end{split}
\end{equation*}
Secondly 
\begin{equation*}
    \begin{split}
        \langle \Lambda_{1}\vect{\phi_{1}},\vect{\phi_{2}}\rangle = \begin{pmatrix}
\alpha \vect{1}_{m_{1}}\\
0
\end{pmatrix}\cdot \begin{pmatrix}
\vect{u}\\
\vect{v}
\end{pmatrix}=\alpha \vect{1}_{m_{1}}\cdot \vect{u}=0.
    \end{split}
\end{equation*}
Since $\Lambda_{1}$ is self-adjoint, $\langle \vect{\phi_{1}},\Lambda_{1}\vect{\phi_{2}}\rangle =\langle \Lambda_{1}\vect{\phi_{1}},\vect{\phi_{2}}\rangle =0.$ Also, since $\vect{u}\cdot D_{1}^{1\slash 2}\vect{1}_{m_{1}}=0,$ we have by the Courant-Fischer Theorem:
\begin{equation*}
        \langle \Lambda_{1}\vect{\phi_{2}},\vect{\phi_{2}}\rangle=\langle \Lambda_{1}'\vect{u},\vect{u}\rangle \leq \mu_{2}(\Lambda_{1})\norm{u}^{2}=c_{1} \norm{u}^{2}\leq c_{1} \norm{\phi_{2}}^{2}=c_{1}(1-\gamma ),
        \end{equation*}
where $\Lambda_{1}'$ is $D(G_{1})^{-1\slash 2}A(G_{1})D(G_{1})^{-1\slash 2}$, with $G_{1}$ considered as a graph on the vertex set $V_{1}.$

We now perform the same calculations for $\Lambda_{2}$. Firstly, for some matrix $B_{2}$
\begin{equation*}
\begin{split}
    \Lambda_{2}\vect{\phi_{1}}&=\dfrac{1}{\sqrt{d_{1}d_{2}}}\begin{pmatrix}0&B_{2}\\
    B_{2}^{T}&0\end{pmatrix}\begin{pmatrix}
\alpha \vect{1}_{m_{1}}\\
\beta \vect{1}_{m_{2}}
\end{pmatrix}=\begin{pmatrix}
\beta \sqrt{\dfrac{d_{1}}{d_{2}}}\vect{1}_{m_{1}}\\
\alpha\sqrt{\dfrac{d_{2}}{d_{1}}} \vect{1}_{m_{2}}
\end{pmatrix}=\begin{pmatrix}
\beta \sqrt{\dfrac{m_{2}}{m_{1}}}\vect{1}_{m_{1}}\\
\alpha\sqrt{\dfrac{m_{1}}{m_{2}}} \vect{1}_{m_{2}}
\end{pmatrix},
\end{split}
\end{equation*}
so that 
\begin{equation*}
    \begin{split}
        \langle \Lambda_{2}\vect{\phi_{1}},\vect{\phi_{1}}\rangle = \begin{pmatrix}
\beta \sqrt{\dfrac{m_{2}}{m_{1}}}\vect{1}_{m_{1}}\\
\alpha\sqrt{\dfrac{m_{1}}{m_{2}}} \vect{1}_{m_{2}}
\end{pmatrix}\cdot \begin{pmatrix}
\alpha \vect{1}_{m_{1}}\\
\beta \vect{1}_{m_{2}}\
\end{pmatrix}
=2\alpha \beta \sqrt{m_{1}m_{2}}.
    \end{split}
\end{equation*}
Next, by the Courant-Fischer Theorem, $\langle \Lambda_{2}\vect{\phi_{2}},\vect{\phi_{2}}\rangle\leq c_{2} \norm{\vect{\phi_{2}}}^{2}=c_{2}(1-\gamma)$ (since $\vect{\phi_{2}}\cdot D_{2}^{1\slash 2 }\vect{1}_{m_{1}+m_{2}}=0$). Furthermore,

\begin{equation*}
    \begin{split}
        \langle\Lambda_{2}\vect{\phi_{1}}, \vect{\phi_{2}}\rangle&=\begin{pmatrix}
\beta \sqrt{\dfrac{m_{2}}{m_{1}}}\vect{1}_{m_{1}}\\
\alpha\sqrt{\dfrac{m_{1}}{m_{2}}} \vect{1}_{m_{2}}
\end{pmatrix}\cdot \begin{pmatrix}
\vect{u}\\
\vect{v}
\end{pmatrix} = 
\beta\sqrt{\dfrac{m_{2}}{m_{1}}}\vect{1}_{m_{1}}\cdot \vect{u}+ \alpha\sqrt{\dfrac{m_{1}}{m_{2}}}\vect{1}_{m_{2}}\cdot \vect{v} =0,
\end{split}
\end{equation*}
since $\vect{1}_{m_{1}}\cdot \vect{u}=\vect{1}_{m_{2}}\cdot \vect{v}=0.$ Finally, since $\Lambda_{2}$ is symmetric and hence self-adjoint, we see that
\begin{equation*}
    \langle \Lambda_{2}\vect{\phi_{2}},\vect{\phi_{1}}\rangle =\langle \vect{\phi_{2}},\Lambda_{2}\vect{\phi_{1}}\rangle=0.
\end{equation*}
We can perform similar calculations for $\Lambda_{3}$.
Putting this all together, we have
\begin{equation*}
\begin{split}
  \langle \Lambda_{1}\vect{\phi}, \vect{\phi}\rangle &\leq   \alpha^{2}m_{1}+c_{1}(1-\gamma),
\end{split}
\end{equation*}
\begin{equation*}
\begin{split}
  \langle \Lambda_{2}\vect{\phi}, \vect{\phi}\rangle& \leq   2\alpha \beta \sqrt{m_{1}m_{2}}+ c_{2}(1-\gamma),
\end{split}
\end{equation*}
\begin{equation*}
\begin{split}
  \langle \Lambda_{3}\vect{\phi}, \vect{\phi}\rangle& \leq    2\alpha \beta \sqrt{m_{1}m_{2}}+ c_{3}(1-\gamma).
\end{split}
\end{equation*}
We calculate 
\begin{equation*}
\dfrac{1}{\sqrt{2}}\alpha \beta\sqrt{m_{1}m_{2}}=  -\alpha^{2}\sqrt{\dfrac{m_{1}}{m_{2}}}\sqrt{m_{1}m_{2}}=-\alpha^{2}m_{1}.
\end{equation*}
Therefore 
\begin{equation*}
\begin{split}
  \langle \Lambda\vect{\phi}, \vect{\phi}\rangle &=\dfrac{1}{2}\langle \Lambda_{1}\vect{\phi}, \vect{\phi}\rangle+\dfrac{1}{2\sqrt{2}}\langle \Lambda_{2}\vect{\phi}, \vect{\phi}\rangle+\dfrac{1}{2\sqrt{2}}\langle \Lambda_{3}\vect{\phi}, \vect{\phi}\rangle\\
  &\leq \dfrac{1}{2}c_{1}(1-\gamma)+ \dfrac{1}{2}\alpha^{2}m_{1}+\dfrac{1}{\sqrt{2}}\alpha \beta\sqrt{m_{1}m_{2}}+ \dfrac{1}{2\sqrt{2}}c_{2}(1-\gamma)
  +\dfrac{1}{\sqrt{2}}\alpha \beta\sqrt{m_{1}m_{2}}\\ &\hspace*{20 pt}+\dfrac{1}{2\sqrt{2}}c_{3}(1-\gamma)\\
  &=\dfrac{1}{2}\alpha^{2}m_{1}-2\alpha^{2}m_{1}+\dfrac{1-\gamma }{2\sqrt{2}}(\sqrt{2}c_{1}+ c_{2}+c_{3})\\
  &=\dfrac{-3}{2}\alpha^{2}m_{1}+\dfrac{1-\gamma }{2\sqrt{2}}(\sqrt{2}c_{1}+ c_{2}+c_{3})\\
  &=-\dfrac{1}{2}\gamma +\dfrac{1-\gamma }{2\sqrt{2}}(\sqrt{2}c_{1}+ c_{2}+c_{3})\\
  &\leq \dfrac{\sqrt{2}c_{1}+ c_{2}+c_{3} }{2\sqrt{2}},
  \end{split}
\end{equation*}
since $0\leq \gamma \leq 1.$
  
  As $\vect{\phi}$ was chosen with $\mu_{2}(\Lambda)=\langle \Lambda\vect{\phi}, \vect{\phi}\rangle$, we see that 
$\mu_{2}(\Lambda)\leq \dfrac{\sqrt{2}c_{1}+c_{2}+c_{3}}{2\sqrt{2}},$ and hence
$$\lambda_{1}(G)=1-\mu_{2}(\Lambda)\geq 1-\dfrac{\sqrt{2}c_{1}+c_{2}+c_{3} }{2\sqrt{2}}.$$

\end{proof}

 \begin{lemma}\label{cor: eigenvalue of union of three graphs}
 Let $G_{i}$, $c_{i}$ be as above. Suppose $c_{1}=\epsilon, c_{2}=c_{3}=\epsilon +\frac{1}{3} $ for some $\epsilon<1\slash 100.$ Then 
 $$\lambda_{1}(G_{1}\cup G_{2}\cup G_{3})\geq\dfrac{3}{4}.$$
 \end{lemma}
 \begin{proof}
 We may apply Lemma \ref{thm: eigenvalue of union of three graphs} to deduce that 
 \begin{equation*}
 \begin{split}
     \lambda_{1}(G_{1}\cup G_{2}\cup G_{3})&\geq 1- \dfrac{(\sqrt{2}+2)\epsilon +2\slash 3}{2\sqrt{2}}\\
     & \geq 1- \dfrac{2\slash 3 +(2+\sqrt {2})\slash 100}{2\sqrt{2}}\\
     &\geq \dfrac{3}{4}.
     \end{split}
 \end{equation*}
 \end{proof}

\section{The spectrum of reduced random graphs}\label{sec: Spectral theory of unions of reduced random graphs}
We have almost understood the spectral distribution of $\Del{k}$ for $\langle A_{n}\;\vert\;R\rangle$ in the $\Gamma(n,k,d)$ model. However, there is one small complication which arises from the fact that we insist upon using cyclically reduced words as relators: the random graphs $\Delta_{k}( A_{n}\;\vert\;R)$ will not allow edges between certain types of words. Therefore we need to introduce a slightly altered model of random graphs.

Some of the results contained within this section are already known. Indeed, \cite[Section 11,12]{DRUTU_Mackay} provides far more general results concerning the eigenvalues of reduced random graphs: we provide alternate proofs of the results we require (again we stress that the results of \cite{DRUTU_Mackay} are far more general than the results we obtain) as the proofs provide an introduction to the proof strategies of alternate results we require that are not covered by \cite{DRUTU_Mackay}. We indicate in the text the results already known.
\subsection{Reduced random graphs}

\begin{definition}
Fix $n,l\geq 1,$ and let $0<p<1.$ For $i=1,\hdots , n$, let  $a_{i+n}:=a_{i}^{-1}$, and for $i=1,\hdots , 2n$ let $$S_{i}=\{w_{1}\hdots w_{l}\in \W{l}\;:\;w_{1}=a_{i}\}=\{(w_{1}\hdots w_{l})^{-1}\in \W{l}\;:\;w_{l}=a_{i}^{-1}\}.$$

For $v\in \W{l}$, let $i(v)$ be the unique integer such that $v\in S_{i(v)}$.
The \emph{reduced random graph} $\mathfrak{Red}(n,l,p)$ is the random graph obtained with vertex set $\W{l}$, and edge set constructed as follows. 

Let $i=1,\hdots, 2n.$ For each pair of vertices $v,w\in \W{l}$, add (each of) the directed edges:
\begin{itemize}
    \item $(v,w)$ labelled by $i(v)$ with probability $p(v,w)$,
    \item $(w,v)$, labelled by $i(w)$ with probability $p(w,v)$, where:
\end{itemize}
 $$p(s,t)=\begin{cases}p\mbox{ if }i(s)\neq i(t),\\
 0\mbox{ if }i(s)=i(t).
 \end{cases}$$
\end{definition}
Note that $\vert \W{l}\vert=2n(2n-1)^{l-1}.$ Furthermore, we can break $\mathfrak{Red}(n,l,p)$ into a union of graphs $\mathfrak{R}_{i}$, where for $i=1,\hdots ,2n$ each $\mathfrak{R}_{i}$ is a bipartite graph with vertex set $V_{1}=S_{i}$, $V_{2}=\W{l}\setminus S_{i}$, and each edge is added with probability $p$. Note that $\mathfrak{R}_{i}\sim G((2n-1)^{l-1},(2n-1)^{l},p)$: therefore, for large $p$, a.a.s. each graph $\mathfrak{R}_{i}$ is almost $((2n-1)^{l-1}p,(2n-1)^{l}p)$-regular. Hence for large $p$ a.a.s. the graph $\mathfrak{Red}(n,l,p)$ is almost $2(2n-1)^{l}p$-regular. Next we prove the following.
\begin{lemma}\label{lem: Red extension}
Let $n,l\geq 1$, let $p$ be such that $(2n-1)^{l}p=\Omega_{l}(\log^{6}(2n-1)^{l}))$, and let $G\sim \mathfrak{Red}(n,l,p)$. There exists a random graph $$G'\sim G(2n(2n-1)^{l-1},2p-p^{2})$$ such that $a.a.s.(l),$ 
$$\mu_{1}(A(G)-A(G'))\leq O_{l}\bigg(\max\bigg\{l,(2n-1)^{l}p^{2},\sqrt{(2n-1)^{l-1}p}\bigg\}\bigg).$$
\end{lemma}
\begin{proof}
Let $\Sigma_{i}$ be the random graph with vertex set $S_{i}$ and each edge added with probability $2p(1-p)$, so that $\Sigma_{i}\sim G((2n-1)^{l-1},2p-p^{2})$. By our assumptions on $p$, we see by Theorem \ref{lem: eigenvalue of ER random graph} that a.a.s.$(l)$ for all $i$ (there are 2n such $i$, so we take the intersection of the $2n$ events) $$\max_{j\neq 1}\vert \mu_{j}(A(\Sigma_{i}))\vert\leq O_{l}\bigg(\sqrt {(2n-1)^{l-1}p}\bigg).$$

Let $H=\bigcup\limits_{i}(\mathfrak{R}_{i}\cup \Sigma_{i})$. The probability that at least one edge connects two vertices $v,w\in S_{i}$ is $2p-p^{2}$. If $v\in S_{i}$ and $w\in S_{j}$ for $i\neq j$ the probability that at least one edge connects $v$ and $w$ is $1-(1-p)^{2}=2p-p^{2}.$ Hence, by collapsing duplicate edges in $H$ we obtain $G'\sim G(2n(2n-1)^{l-1},2p-p^{2})$. Next, note that 
$$A(G')=A(G)+\sum A_{i}+K,$$ where $K$ takes into account the double edges obtained from the unions, 
and $A_{i}$ is the adjacency matrix of the graph $G_{i}$ which has vertex set $V(G)$ and edge set $E(\Sigma _{i})$. Since the edge sets of each $\Sigma _{i}$ are pairwise disjoint, one can easily see that $\mu_{1}(-\sum_{i}A_{i})=\max_{i}\mu_{1}(-A_{i})$.

$K$ is the adjacency matrix of a random graph where edges are added with probability $0$ or $p^{2}$. Using the Chernoff bounds for the degrees, we can see that if $(2n-1)^{l}p^{2}=\Omega_{l}(l)$, then a.a.s.$(l)$ $\vert\vert K\vert\vert_{\infty}=O_{l}((2n-1)^{l}p^{2})$. Otherwise, we may deduce that $\vert \vert K\vert \vert_{\infty}=O_{l}(\log (2n-1)^{l})=O_{l}(l)$.

Hence by Weyl's inequality:
\begin{align*}
\mu_{1}(A(G)-A(G'))&=\mu_{1}(-K-\sum A_{i})\\
&\leq \mu_{1}(-K)+\mu_{1}(-\sum A_{i})
\\&=O_{l}\bigg(\max\bigg\{\vert \vert K\vert\vert_{\infty},\mu_{1}\bigg(-\sum A_{i}\bigg)\bigg\}\bigg)\\
&=O_{l}\bigg(\max\bigg\{\vert \vert K\vert\vert_{\infty},\max_{i}\{\mu_{1}(-A_{i})\}\bigg\}\bigg)\\
&\leq O_{l}\bigg(\max\bigg\{l, (2n-1)^{l}p^{2},\sqrt{(2n-1)^{l-1}p}\bigg\}\bigg).
    \end{align*}
\end{proof}

Similarly we define the following.
\begin{definition}
Fix $n\geq 1,l\geq 3, 0<p<1.$ Let $a_{i+n}:=a_{i}^{-1}$. For $i=1,\hdots , 2n$, let $$S_{i}'=\{w_{1}\hdots w_{l}\in \W{l}\;:\;w_{1}=a_{i}\},$$ and $$T_{i}'=\{(w_{1}\hdots w_{l+1})^{-1}\in \W{l+1}\;:\;w_{l+1}=a_{i}^{-1}\}.$$
The \emph{reduced random bipartite graph} $\mathfrak{BRed}(n,l,p)$ is the random graph with vertex set $V_{1}=\W{l},$ $V_{2}=\W{l+1}$, and for each $v\in S_{i}'$ and vertex $w\in V_{2}- T_{i}'$, the edge $(v,w)$ is added with probability $p$. The graph $\mathfrak{BR}_{i}$ is the random bipartite graph obtained as a subgraph with vertex set $V_{1}=S_{i}'$ and $V_{2}=\W{l+1}\setminus T_{i}'.$ \end{definition} Again, for large $p$ the graph $\mathfrak{BRed}(n,l,p)$ is almost $((2n-1)^{l+1}p,(2n-1)^{l}p)$-regular. We can approximate this graph by an Erd\"os--R\'enyi random bipartite graph, similarly to the case of $\mathfrak{Red}(n,l,p).$

\begin{lemma}\label{lem: extending bired}
Let $G\sim \mathfrak{BRed}(n,l,p)$, where $(2n-1)^{l}p=\Omega_{l}(\log (2n-1)^{l})$. There exists a random graph $G'\sim G(2n(2n-1)^{l-1},2n(2n-1)^{l},p)$ such that $a.a.s.(l),$ 
$$\mu_{1}(A(G)-A(G'))\leq (1+o_{l}(1))(2n-1)^{l-1\slash 2}p.$$

\end{lemma}
\begin{proof}
This follows similarly to the proof of Lemma \ref{lem: Red extension} for $\mathfrak{Red}(n,l,p)$.

For $i=1,\hdots , 2n$, let $\Sigma_{i}$ be the random graph with vertex set $V_{1}=S_{i}, V_{2}=T_{i}$ and each edge added with probability $p$, so that $\Sigma_{i}\sim G((2n-1)^{l-1},(2n-1)^{l},p)$. 

Then $$G'=G\cup\bigcup_{i}\Sigma_{i}\sim G(2n(2n-1)^{l-1},2n(2n-1)^{l},p).$$

We see that $\mu_{1}(A(G)-A(G'))=\mu_{1}(-\sum_{i}A_{i})$, where $A_{i}$ is the adjacency matrix of the graph with vertex set $V(G)$ and edge set $E(\Sigma_{i})$. Since the edge sets of the $\Sigma_{i}$ are pairwise disjoint, (and the graphs are bipartite, so their spectrum is symmetric around $0$) we see that $$\mu_{1}(-\sum_{i}A_{i})\leq \max_{i}\mu_{1}(-A_{i})=\max_{i}\mu_{1}(A_{i})\leq (1+o_{l}(1))(2n-1)^{l-1\slash  2}p,$$ by Lemmas \ref{lem: max eigenvalue of adjacency matrix} and \ref{lem: almost regular G(m,n,p)}.
\end{proof}

We may analyse the eigenvalues of reduced random graphs, as follows.

\begin{lemma}\label{lem: eigenvalue of red}\cite[Theorem 11.8, 11.9]{DRUTU_Mackay}
Let $n\geq 2$, and $p$ be such that $p=o_{l}(1)$ and $(2n-1)^{l}p=\Omega_{l}(l^{6}).$ Let $G\sim \mathfrak{Red}(n,l,p)$. Then a.a.s.$(l)$ $\lambda_{1}(G)\geq 1-o_{l}(1).$
\end{lemma}
\begin{proof}
Let $G'$ be the graph from Lemma \ref{lem: Red extension}, so that $G'\sim G(2n(2n-1)^{l-1},2p-p^{2})$ and $$\mu_{1}(A(G)-A(G'))\leq O_{l}\bigg(\max\bigg\{l,(2n-1)^{l}p^{2},\sqrt{(2n-1)^{l-1}p}\bigg\}\bigg).$$ Let $D'=D(G')$, and $A'=A(G')$. Note that $G$ is almost $2(2n-1)^{l}p$-regular, and hence, 
\begin{align*}
  \mu_{1}(D^{-1\slash 2}(A-A')D^{-1\slash 2})&\leq O_{l}\bigg(\dfrac{1+o_{l}(1)}{(2n-1)^{l}p}\max\bigg\{l,(2n-1)^{l}p^{2},\sqrt{(2n-1)^{l-1}p}\bigg\}\bigg)\\
  &= o_{l}(1).      
    \end{align*}

Next, by our assumption on p, $$2n(2n-1)^{l}p=\Omega_{l}(l^{6})=\Omega_{l}\bigg(\log^{6} 2n(2n-1)^{l-1}\bigg),$$
so that by Theorem \ref{lem: eigenvalue of ER random graph}, a.a.s.$(l)$, $$\mu_{2}\bigg(D'^{-1\slash 2}A'D'^{-1\slash 2}\bigg)=o_{l}(1).$$
Next, $D(G)^{-1\slash 2}AD(G)^{-1\slash 2}=\frac{(2-p)n}{2n-1}D'^{-1\slash 2}A'D^{-1\slash 2}+K$, where $\vert \vert K\vert \vert_{\infty}=o_{l}(1)$. Hence $\mu_{1}(K)=o_{l}(1)$. Therefore, by Theorem \ref{lem: eigenvalue of ER random graph} and Weyl's inequality, $a.a.s.(l)$
\begingroup
\allowdisplaybreaks
\begin{align*}
\mu_{2}(D^{-1\slash 2}AD^{-1\slash 2})&=\mu_{2}(D^{-1\slash 2}A'D^{-1\slash 2}+D^{-1\slash 2}AD^{-1\slash 2}-D^{-1\slash 2}A'D^{-1\slash 2})\\
        &\leq \mu_{2}(D^{-1\slash 2}A'D^{-1\slash 2})+\mu_{1}(D^{-1\slash 2}(A-A')D^{-1\slash 2})\\
        &=\mu_{2}\bigg(\frac{(2-p)n}{2n-1}D'^{-1\slash 2}A'D'^{-1\slash 2}+K\bigg)+o_{l}(1)\\
          &\leq \frac{(2-p)n}{2n-1}\mu_{2}\bigg(D'^{-1\slash 2}A'D'^{-1\slash 2}\bigg)+\mu_{1}(K)+o_{l}(1)\\
        &\leq \frac{(2-p)n}{2n-1}\mu_{2}\bigg(D'^{-1\slash 2}A'D'^{-1\slash 2}\bigg)+o_{l}(1)\\
        &=o_{l}(1).
\end{align*}
\endgroup
The result follows by Remark \ref{rmk: switching between eigenvalues of graphs}.
\end{proof}

\begin{lemma}\label{lem: eigenvalue of bired}
Let $n\geq 2$, and $p$ be such that $p=o_{l}(1)$ and $(2n-1)^{l}p=\Omega_{l}(l^{6}).$ Let $G\sim \mathfrak{BRed}(n,l,p)$. Then a.a.s.$(l)$ $$\lambda_{1}(G)\geq 1-1\slash(2n-1)-o_{l}(1).$$
\end{lemma}
We note that we cannot prove that the above bound is sharp, but it is sufficient for our needs.
\begin{proof}
Let $G'$ be the graph from Lemma \ref{lem: extending bired}, so that $G'\sim G(2n(2n-1)^{l-1},2n(2n-1)^{l},p)$, and $\mu_{1}(A(G)-A(G'))\leq (1+o_{l}(1))(2n-1)^{l-1\slash 2}p.$ By Lemma \ref{lem: extending bired}, $$\mu_{1}(D^{-1\slash 2}(A-A')D^{-1\slash 2})\leq [1+o_{l}(1)]\dfrac{1}{2n-1}.$$
Next, $D(G)^{-1\slash 2}A' D^{-1\slash 2}=\frac{2n}{2n-1}D'^{-1\slash 2}A'D'^{-1\slash 2}+K$, where $$K=\begin{pmatrix}
0&H\\
H^{T}&0
\end{pmatrix}$$ and $\sqrt{\vert \vert H\vert\vert_{\infty}\vert \vert H\vert\vert_{1}}=o_{l}(1)$. Hence $\mu_{1}(K)=o_{l}(1)$. Therefore, by Theorem \ref{lem: eigenvalue of ER bipartite random graph}, and using Remark \ref{rmk: switching between eigenvalues of graphs} and Weyl's inequalities similarly to the proof of Lemma \ref{lem: eigenvalue of red},
\begin{align*}
        \mu_{2}(D^{-1\slash 2}AD^{-1\slash 2})&=\mu_{2}(D^{-1\slash 2}A'D^{-1\slash 2}+D^{-1\slash 2}AD^{-1\slash 2}-D^{-1\slash 2}A'D^{-1\slash 2})\\
        &\leq \mu_{2}(D^{-1\slash 2}A'D^{-1\slash 2})+\mu_{1}(D^{-1\slash 2}(A-A')D^{-1\slash 2})\\
        &\leq\mu_{2}\bigg(\frac{2n}{2n-1}D'^{-1\slash 2}A'D'^{-1\slash 2}+K\bigg)+\dfrac{1}{2n-1}+o_{l}(1)\\
        &\leq  \frac{2n}{2n-1}\mu_{2}(D'^{-1\slash 2}A'D'^{-1\slash 2})+\mu_{1}(K)+\dfrac{1}{2n-1}+o_{l}(1)\\
        &=\dfrac{1}{2n-1}+o_{l}(1).
    \end{align*}
The result follows by Remark \ref{rmk: switching between eigenvalues of graphs}.
\end{proof}

\subsection{Regular subgraphs of random graphs}
We now need an auxiliary result concerning regular subgraphs of random graphs. Recall that a subgraph $H$ of $G$ is \emph{spanning} if $V(H)=V(G).$ We first note the following.
\begin{theorem*}\cite{shamirupfal}\label{lem: large regular factors in er graph}
Suppose $mp=\omega(m)\log(m)$ for some $\omega (m)\rightarrow\infty. $ Let $\delta\geq \omega^{-\theta}$ for some $0<\theta< 1\slash 2$, and let $G\sim G(m,p)$. Then $a.a.s.(m)$, $G$ contains a $(1-\delta)mp$-regular spanning subgraph.
\end{theorem*}

We wish to prove the analogue for random bipartite graphs. We do this similarly to \cite[Theorem 1.4]{ferberkirvelevichsudakov}, which proves the result in the regime $m_{1}=m_{2}$.

\begin{theorem*}\cite[Theorem 1.4]{ferberkirvelevichsudakov}
Let $m\geq 1$ and $p=p(m)>0$ be such that $mp=\omega (m)\log m$ for some $\omega \rightarrow\infty $ as $m\rightarrow\infty$. Let $\delta \geq \omega^{-\theta}$ for some $\theta <1\slash 2$, and $G\sim G(m,m,p)$. Then a.a.s.$(m)$ $G$ contains a $((1-\delta)mp,(1-\delta)mp)$-regular spanning subgraph.
\end{theorem*}
In the $k$-angular model, we have $m_{1}=m_{2}\slash n$, where $n\rightarrow\infty$, so we need to extend the above to a more general setting. We will use the following theorem, commonly known as the Ore-Reyser theorem: see for example \cite{orefactorthm} or Tutte \cite{Tuttefactors}. Recall that for a graph $G$, and disjoint sets $A,B\subseteq V(G)$, we define $e_{G}(A,B)$ to be the number of edges in $G$ between the sets $A$ and $B$.

\begin{theorem*}[Ore-Reyser Theorem]
Let $G$ be a bipartite graph and let $d_{1},d_{2}
\geq 0$. $G$ contains a $(d_{1},d_{2})$-regular spanning subgraph if and only if $d_{1}\vert V_{1}\vert =d_{2}\vert V_{2}\vert$, and for all $A\subseteq V_{1}$ and $B\subseteq V_{2}$: $d_{1}\vert A\vert \leq e_{G}(A,B)+d_{2}(\vert V_{2}\vert-\vert B\vert).$
\end{theorem*}

Using the above, we can prove the following: this follows almost identically to the proof of \cite[Theorem 1.4]{ferberkirvelevichsudakov}, with very minor changes.

\begin{theorem}\label{thm: regular factor in random bipartite graphs}
Let $m_{2}=m_{2}(m_{1})\geq m_{1}$ and let $p=p(m_{1})>0$ be such that $m_{1}p=\omega (m_{1})\log m_{2}$ for some $\omega \rightarrow\infty $ as $m_{1}\rightarrow\infty$. Let $\delta \geq \omega^{-\theta}$ for some $\theta <1\slash 2$, and $G\sim G(m_{1},m_{2},p)$. Then a.a.s.$(m_{1})$ $G$ contains a $((1-\delta)m_{2}p,(1-\delta)m_{1}p)$-regular spanning subgraph with probability greater than $1-m_{2}^{-\Omega_{m_{1}}(1)}$.
\end{theorem}
Again, the proof of this follows extremely similarly to the proof of \cite[Theorem 1.4]{ferberkirvelevichsudakov}; we include it for completeness.
\begin{proof}

Let $d_{1}=(1-\delta)m_{2}p$ and $d_{2}=(1-\delta)m_{1}p$. We wish to prove that a.a.s.$(m_{1})$ for all $A\subseteq V_{1}$ and $B\subseteq V_{2}$:
\begin{align*}
   0&\leq  e_{G}(A,B)+d_{2}(m_{2}-\vert B\vert )-d_{1}\vert A\vert \\
    &=e_{G}(A,B)+d_{1}(m_{1}-\vert A\vert -m_{1}\vert B\vert \slash m_{2}).
\end{align*}
If we are able to prove this, then we may conclude the desired result by the Ore-Reyser theorem. 
Note that if $\vert A\vert +m_{1}\vert B\vert \slash m_{2}\leq m_{1}$ then we are immediately finished. 
Let us suppose otherwise; we now analyse different cases. 

 To begin, let $n_{1}:=m_{1}\slash \log\log m_{1}$. We may now assume that $\vert A\vert +m_{1}\vert B\vert \slash m_{2}>m_{1}$. Suppose first that $\vert A\vert \leq n_{1}$, then $(m_{2}(m_{1}-\vert A\vert)\slash m_{1})+1\leq \vert B\vert \leq m_{2}$. Note that $e_{G}(A,B)$ has the distribution $Bin(\vert A\vert \vert B \vert ,p)$. We may apply the Chernoff bounds to deduce that 
 
 $$\mathbb{P}(e_{G}(A,B)\leq (1-\delta )\vert A\vert \vert B\vert p)\leq \exp\bigg(\dfrac{-\delta^{2}\vert A\vert \vert B\vert p}{2}\bigg).$$
 For $$\vert A\vert =a\leq n_{1},  \vert B\vert = b\geq \dfrac{m_{2}(m_{1}-a)}{m_{1}},$$ and $m_{1}$ sufficiently large, this is bounded above by 
 $$\exp\bigg(\dfrac{-\delta^{2}a m_{2}(m_{1}-a) p\slash m_{1}}{2}\bigg)\leq \exp\bigg(am_{2}p\dfrac{-\delta^{2}(m_{1}-n_{1})}{2m_{1}}\bigg)\leq \exp \bigg(-\delta^{2} \dfrac{m_{2}ap}{4}\bigg).$$ 
 Therefore the probability that there exists such sets with $e_{G}(A,B)\leq (1-\delta)\vert A\vert\vert B\vert p$ is bounded above by
\begin{align*}
        \sum\limits_{a=1}^{n_{1}}\sum\limits_{b=(m_{2}(m_{1}-a)\slash m_{1})+1}^{m_{2}}\binom{m_{1}}{a}\binom{m_{2}}{b}e^{-\delta^{2}m_{2}ap\slash 4}
        &= \sum\limits_{a=1}^{n_{1}}\sum\limits_{b=1}^{m_{2}a\slash m_{1}}\binom{m_{1}}{a}\binom{m_{2}}{b}e^{-\delta^{2} m_{2}ap\slash 4}\\
     {\small  \bigg (\mbox{using }\binom{m_{2}}{b}\leq \binom{m_{2}}{m_{2}a\slash m_{1}}\mbox{ for }b\leq m_{2}a\slash m_{1}\bigg)} &\leq \sum\limits_{a=1}^{n_{1}} \frac{m_{2}a}{m_{1}}\binom{m_{1}}{a}\binom{m_{2}}{m_{2}a\slash m_{1}}e^{-\delta^{2}m_{2}ap\slash 4} \\
      {\small  \bigg (\mbox{using }\binom{m_{1}}{a}\leq \binom{m_{2}}{m_{2}a\slash m_{1}}\mbox{ as }m_{2}\slash m_{1}\geq 1\bigg)}   &\leq \sum\limits_{a=1}^{n_{1}} \frac{m_{2}a}{m_{1}}\binom{m_{2}}{m_{2}a\slash m_{1}}^{2}e^{-\delta^{2} \frac{m_{2}}{m_{1}}am_{1}p\slash 4}\\
        &\leq m_{2} \sum\limits_{a=1}^{n_{1}}\bigg(\dfrac{m_{2}^{2}e^{2}}{m_{2}^{2}a^{2}\slash m_{1}^{2}}e^{-\Omega( \log m_{2})}\bigg)^{\frac{am_{2}}{m_{1}}}\\
        &=m_{2}^{-\Omega_{m_{1}}(1)},
    \end{align*}
since $\delta^{2}m_{1}p\geq \omega^{1-2\theta }\log m_{2}$ for some $\theta <1\slash 2.$ The case is similar for $\vert B\vert \leq n_{2}:=m_{2}\slash \log\log m_{2}$. Next we may assume that $\vert A\vert\geq n_{1}$ and that $\vert B\vert \geq n_{2}$. First assume that $\vert A\vert \leq m_{1}\vert B\vert\slash m_{2},$ so that $\vert B\vert \geq m_{2}\slash 2$. The probability there exists such $A,B$ with $e_{G}(A,B)\leq (1-\delta)\vert A\vert\vert B\vert p$ is bounded above by

\begin{align*}
         \sum\limits_{a=n_{1}}^{m_{1}}\sum\limits_{b=m_{2}\slash 2}^{m_{2}}\binom{m_{1}}{a}\binom{m_{2}}{b}e^{-\delta^{2} abp\slash 2}
         &\leq  \sum\limits_{a=n_{1}}^{m_{1}}\sum\limits_{b=m_{2}\slash 2}^{m_{2}}\binom{m_{1}}{a}\binom{m_{2}}{b} e^{-\delta^{2} n_{1}m_{2}p\slash 4}\\
         &\leq 2^{m_{1}+m_{2}}e^{-\delta^{2} m_{1}m_{2}p\slash (4\log \log m_{1})}\\
         &\leq m_{2}^{-\Omega_{m_{1}}(1)},
    \end{align*}
since $\delta^{2}m_{1}p\slash \log \log m_{1}\geq \omega^{1-2\theta}\log m_{2}\slash \log \log m_{1}=\Omega_{m_{1}} (1).$
Similarly, if $\vert A\vert \geq m_{1}\vert B\vert \slash m_{2}$, the probability that there exists $A,B$ with $e_{G}(A,B)\leq (1-\delta)\vert A\vert\vert B\vert p$ is bounded above by
\begin{align*}
         \sum\limits_{b=n_{2}}^{m_{2}}\sum\limits_{a=m_{1}\slash 2}^{m_{1}}\binom{m_{1}}{a}\binom{m_{2}}{b}e^{-\delta^{2} abp\slash 2}&\leq  \sum\limits_{b=n_{2}}^{m_{2}}\sum\limits_{b=m_{1}\slash 2}^{m_{1}}\binom{m_{1}}{a}\binom{m_{2}}{b} e^{-\delta^{2} n_{2}m_{1}p\slash 4}\\
         &\leq 2^{m_{1}+m_{2}}e^{-\delta^{2} m_{1}m_{2}p\slash (4\log \log m_{2})}\\
         &\leq m_{2}^{-\Omega_{m_{1}}(1)},
    \end{align*}
since $\delta ^{2} m_{1}p=\Omega_{m_{1}} (\log m_{2})$.

Now, consider $A\subseteq V_{1}, B\subseteq V_{2}.$  If $\vert A\vert +m_{1}\vert B\vert \slash m_{2}\leq m_{1}$, then it is immediate that $$0\leq  e_{G}(A,B)+d_{1}(m_{1}-\vert A\vert -m_{1}\vert B\vert \slash m_{2}).$$
  Otherwise, we have proved that a.a.s.$(m_{1})$ $e_{G}(A,B)\geq (1-\delta)\vert A\vert \vert B\vert p$, so that
a.a.s.$(m_{1})$
\begin{align*}
    &e_{G}(A,B)+d_{1}(m_{1}-\vert A\vert -m_{1}\vert B\vert \slash m_{2})\\
    &\geq (1-\delta)\vert A\vert \vert B\vert p +(1-\delta)m_{2}p(m_{1}-\vert A\vert -m_{1}\vert B\vert \slash m_{2})\\
&=(1-\delta) \vert A\vert \vert B\vert p +(1-\delta)m_{1}m_{2}p-(1-\delta)\vert A\vert m_{2}p -(1-\delta)m_{1}\vert B\vert p\\
    &=(1-\delta)p(\vert A\vert \vert B\vert+m_{1}m_{2}-\vert A\vert m_{1}-\vert B\vert m_{2})\\
    &=(1-\delta)p(m_{1}-\vert A\vert )(m_{2}-\vert B\vert)\\
    &\geq 0,
    \end{align*}
since $\vert A\vert\leq m_{1}$ and $\vert B\vert\leq m_{2}$.
The result now follows by the Ore-Reyser theorem.
\end{proof}

\subsection{Regular subgraphs in reduced random graphs}

Finally, we need to address the issue of vertex degrees: in order to use Lemma \ref{lem: eigenvalue of union of three graphs on same vertex set} and Theorem \ref{thm: eigenvalue of union of three graphs}, we need our graphs to be regular, and to have large eigenvalue. Therefore we need to show that $\mathfrak{Red}(n,l,p)$, $\mathfrak{BRed}(n,l,p)$ contain regular spanning subgraphs with large first eigenvalue.

\begin{lemma}\label{lem: spectra of regular reduced random graphs}
Let $n\geq 2$, and let $p$ be such that $(2n-1)^{l}p=\Omega_{l}(\log^{6} (2n-1)^{l+1})=\Omega_{l}(l^{6})$ and $p=o_{l}(1).$ Let $G_{1}\sim \mathfrak{Red}(n,l,p)$ and $G_{2}\sim \mathfrak{Bred}(n,l,p)$. There exists $\epsilon= \epsilon (p)=o_{l}(1)$ such that for all $o_{l}(1)=\delta\geq \epsilon$, a.a.s.$(l)$ there exist spanning subgraphs $H_{i}\leq G_{i}$ such that
\begin{enumerate}[label=$\roman*)$]
    \item $H_{1}$ is $2(1-\delta)(2n-1)^{l}p$-regular, with $\lambda_{1}(H_{1})\geq 1-o_{l}(1),$
    \item and $H_{2}$ is $((1-\delta)(2n-1)^{l+1}p,(1-\delta)(2n-1)^{l}p)$-regular, with\\ $\lambda_{1}(H_{2})\geq 1-\dfrac{1}{2n-1} +o_{l}(1).$
\end{enumerate}
\end{lemma}
\begin{proof}
The first part of $i)$ and $ii)$, i.e. the existence of the regular subgraphs, follows from \cite{shamirupfal} and Lemma \ref{thm: regular factor in random bipartite graphs}. In particular for such a random graph $G_{1}$, and for $i=1,\hdots , n$, the random graph $\mathfrak{R}_{i}$ contains a $((1-\delta)(2n-1)^{l}p,(1-\delta)(2n-1)^{l-1}p))$-regular spanning subgraph $SP_{i}$ with probability at least $$1-(2n-1)^{-l\omega(l)}$$ for some $\omega=\Omega_{l}(1)$. Therefore the probability that all the graphs $\mathfrak{R}_{i}$ contain a spanning subgraph is at least $$\left(1-(2n-1)^{-l\omega(l)}\right)^{n}=1-o_{l}(1)$$ Taking $H_{1}=\cup_{i}SP_{i}$, the result on regular subgraphs follows. 

By \cite[Lemma 4.5]{kotowski} and Lemma \ref{lem: spectra under addition of small bipartite graph}, $\lambda_{1}(H_{i})=\lambda_{1}(G_{i})+o_{l}(1)$, since the $G_{i}$ is formed from $H_{i}$ by the addition of graphs of suitably small degrees. The result follows by Lemmas \ref{lem: eigenvalue of red} and \ref{lem: eigenvalue of bired}.
\end{proof}

Similarly, we can prove the following.
\begin{lemma}\label{lem: spectra of regular reduced random graphs $l$ growing}
Let $n\geq 2,l\geq 5$, and let $p$ be such that $(2n-1)^{l}p=\Omega_{n}(\log^{6} (2n-1)^{l+1})=\Omega_{n}(\log^{6}(2n-1))$ and $p=o_{n}(1).$ Let $G_{1}\sim \mathfrak{Red}(n,l,p)$ and $G_{2}\sim \mathfrak{Bred}(n,l,p)$. There exists $\epsilon= \epsilon (p)=o_{n}(1)$ such that for all $o_{n}(1)=\delta\geq \epsilon$, a.a.s.$(n)$ there exist spanning subgraphs $H_{i}\leq G_{i}$ such that
\begin{enumerate}[label=$\roman*)$]
    \item $H_{1}$ is $2(1-\delta)(2n-1)^{l}p$-regular, with $\lambda_{1}(H_{1})\geq 1-o_{n}(1),$
    \item and $H_{2}$ is $((1-\delta)(2n-1)^{l+1}p,(1-\delta)(2n-1)^{l}p)$-regular, with\\ $\lambda_{1}(H_{2})\geq 1-\dfrac{1}{2n-1} +o_{n}(1).$
\end{enumerate}
\end{lemma}
\begin{proof}
This is extremely similar to the previous lemma.

The first part of $i)$ and $ii)$, i.e. the existence of the regular subgraphs, follows from \cite{shamirupfal} and Lemma \ref{thm: regular factor in random bipartite graphs}. In particular for such a random graph $G_{1}$, and for $i=1,\hdots , n$, the random graph $\mathfrak{R}_{i}$ contains a $((1-\delta)(2n-1)^{l}p,(1-\delta)(2n-1)^{l-1}p))$-regular spanning subgraph $SP_{i}$ with probability at least $$1-(2n-1)^{-l\omega(n))}$$ for some $\omega=\Omega_{n}(1)$. Therefore the probability that all the graphs $\mathfrak{R}_{i}$ contain a spanning subgraph is at least $$\left(1-(2n-1)^{-l\omega(n)}\right)^{n}=1-o_{n}(1).$$ Taking $H_{1}=\cup_{i}SP_{i}$, the result on regular subgraphs follows. 

In the case of growing $n$, the graphs $\mathfrak{R}(n,l,p)$ and $\mathfrak{BRed}(n,l,p)$ have a very small proportion of disallowed edges so have eigenvalues extremely close to those of an (bipartite) Erd\"os--R\'enyi random graph. The result then follows from Theorems \ref{lem: eigenvalue of ER random graph} and \ref{lem: eigenvalue of ER bipartite random graph}.
\end{proof}

\section{Property (T) in random quotients of free groups}\label{sec: Property (T) in random quotients of free groups}
Finally, we may prove Theorems \ref{mainthm: property t k-angular model} and \ref{mainthm: property t density model}. We in fact provide the full proof for Theorem \ref{mainthm: property t density model}, as this is the harder of the two theorems to prove, and indicate how to alter the proof of this theorem in order to prove Theorem \ref{mainthm: property t k-angular model}. However, we first define a slightly different model of random groups.

\begin{definition}
Let $n\geq 2$, $k \geq 3$, and let $0<p=p(n,k)<1$. The random group model $\Gamma_{p}(n,k,p)$ is the model obtained as following. We let $\Gamma =\langle A_{n}\;\vert\;R\rangle,$ where $R$ is obtained by adding each word in $\mathcal{C}(n,k)$ with probability $p$.
\end{definition}
We in fact prove the following theorem.

\begin{theorem}\label{thm: Gp has Property (T)}
Let $n\geq 2$, and let $p$ be such that 
$$(2n-1)^{k\slash 3}p=\Omega_{k}(k^{6}).$$ Let $\Gamma_{k}\sim\Gamma_{p}(n,k,p)$. Then $$\lim_{k\rightarrow \infty }\mathbb{P}(\Gamma_{k}\mbox{ has Property }(T))=1.$$
\end{theorem}
Assuming this, we may prove Theorem \ref{mainthm: property t density model}.
\begin{proof}[Proof of Theorem \ref{mainthm: property t density model}]
Fix $n\geq 2$ and $d>1\slash 3$. Choose $1\slash 3<d'<d$, and let $$\Gamma_{k}'=\langle A_{n}\;\vert\; R'\rangle\sim \Gamma_{p}(n,k,(2n-1)^{kd'-k}).$$ It is easily seen that a.a.s.$(k)$: $$\vert R'\vert=(1+o_{k}(1))(2n-1)^{kd'}.$$ Choose a random subset $R$ with $R'\subseteq R\subseteq \W{k}$ and $\vert R\vert = (2n-1)^{kd}$, and let $\Gamma_{k}=\langle A_{n}\;\vert\;R\rangle.$ Then $\Gamma_{k} \sim \Gamma (n,k,d),$ and there is a clear epimorphism $\Gamma_{k}'\twoheadrightarrow \Gamma_{k}$. Since Property (T) is preserved under epimorphisms, the result follows by Theorem \ref{thm: Gp has Property (T)}.
\end{proof}

Let $\Gamma$ be a random group in the $\Gamma_{p} (n,k,p)$ model. We consider the three cases.
\begin{enumerate}[label=$\mathbf{k=\arabic*\mbox{\bf\; mod }3.}$,start=0]
\item Let $l_{k}=L_{k}=k\slash 3.$ We may define the graphs $\Sigma_{1},\Sigma_{2},\Sigma_{3}$ where: $$V(\Sigma_{1})=V(\Sigma_{2})=V(\Sigma_{3})=\W{k\slash 3},$$ and for each relator $r=r_{x}r_{y}r_{z}$ with $r_{x},r_{y},r_{z}\in\W{k\slash 3}$, we add the edge $(r_{x},r_{z}^{-1})$ to $\Sigma_{1}$, $(r_{y},r_{x}^{-1})$ to $\Sigma_{2}$ and $(r_{z},r_{y}^{-1})$ to $\Sigma_{3}.$ 

\item Let $l_{k}=(k-1)\slash 3$ and $L_{k}=(k+2)\slash 3$. Again, we may write each relator $r=r_{x}r_{y}r_{z}$ for $r_{x}$, $r_{y}\in \W{(k-1)\slash 3}$ and $r_{z}\in \W{(k+2)\slash 3}.$ We again split the graph $\Del{k}$ into $\Sigma_{1}$, $\Sigma_{2}$, $\Sigma_{3}$, where: $$V(\Sigma_{1})=V(\Sigma_{3})=\W{(k-1)\slash 3}\sqcup \W{(k+2)\slash 3},$$ and $V(\Sigma_{2})=\W{(k-1)\slash 3}$. For each relator $r=r_{x}r_{y}r_{z}$, we add the edge $(r_{x},r_{z}^{-1})$ to $\Sigma_{1}$, $(r_{y},r_{x}^{-1})$ to $\Sigma_{2}$, and $(r_{z},r_{y}^{-1})$ to $\Sigma_{3}$. 
\item Let $l_{k}=(k+1)\slash 3$ and $L_{k}=(k-2)\slash 3$. Again, we may write each relator $r=r_{x}r_{y}r_{z}$ for $r_{x}$, $r_{y}\in \W{(k+1)\slash 3}$ and $r_{z}\in \W{(k-2)\slash 3}.$ We again split the graph $\Del{k}$ into $\Sigma_{1}$, $\Sigma_{2}$, $\Sigma_{3}$, where: $$V(\Sigma_{1})=V(\Sigma_{3})=\W{(k-2)\slash 3}\sqcup \W{(k+1)\slash 3},$$ and $V(\Sigma_{2})=\W{(k+1)\slash 3}$. For each relator $r=r_{x}r_{y}r_{z}$, we add the edge $(r_{x},r_{z}^{-1})$ to $\Sigma_{1}$, $(r_{y},r_{x}^{-1})$ to $\Sigma_{2}$, and $(r_{z},r_{y}^{-1})$ to $\Sigma_{3}$. 
\end{enumerate}
Next we show there aren't too many double edges in the graphs $\Sigma_{i}$, similarly to \cite{antoniuktriangle}.

\begin{lemma}\label{lem: double edges form a matching}
Let $n\geq 2$, and let $p$ be such that 
\begin{enumerate}[label=$\roman*)$]
    \item $(2n-1)^{2k-L_{k}}p^{3}=o_{k}(1)$,
    \item and $(2n-1)^{2k+l_{k}}p^{4}=o_{k}(1)$.
\end{enumerate}
Let $\Gamma_{k}\sim\Gamma_{p}(n,k,p)$, and let $\Sigma_{i}$ be described as above. For $i=1,2,3$ a.a.s.$(k)$ there is no pair of vertices $u,v$ with at least three edges between them in $\Sigma_{i}$, and the set of double edges in $\Sigma_{i}$ forms a matching, i.e. the endpoints of the double edges are all distinct. 
\end{lemma}
Note that for $1\slash 3<d<5\slash 12$, $p(d)=(2n-1)^{kd-k}$ satisfies the above conditions.

\begin{proof}
We prove this for $i=1$. Throughout, note that $k=2l_{k}+L_{k}$. The probability, $\mathbb{P}_{3}$, that there exists a pair of vertices $u,v$ with at least three edges between $u$ and $v$ is bounded above by 
\begin{align*}
        \mathbb{P}_{3}&\leq O_{k}\left ((2n-1)^{l_{k}+L_{k}}(2n-1)^{3l_{k}}p^{3}\right)\\
        &=O_{k}\left((2n-1)^{2k-L_{k}}p^{3}\right )\\
        &=o_{k}(1).
    \end{align*}
The probability, $\mathbb{P}_{doub}$ that there are vertices $u,v,w$ with double edges between $u$ and $v$ and $u$ and $w$ is bounded by
\begin{align*}
        \mathbb{P}_{doub}&=O_{k}\left((2n-1)^{l_{k}}(2n-1)^{2L_{k}}(2n-1)^{4l_{k}}p^{4}\right)\\
        &=O_{k}\left((2n-1)^{2k+l_{k}}p^{4}\right)\\
        &=o_{k}(1).
    \end{align*}
\end{proof}
This is sufficient to prove our main theorem.

\begin{proof}[Proof of Theorem \ref{thm: Gp has Property (T)}]
Let $\Gamma_{k}=\langle A_{n}\;\vert\;R\rangle\sim \Gamma_{p}(n,k,p)$, and consider $\Delta_{k}:=\Delta_{k}( A_{n}\;\vert\;R ).$ 
 Since Property (T) is preserved by epimorphisms, we may assume that $p\leq (2n-1)^{kd-k}$ for some $d<4\slash 9$:
for any $1\slash 3<d<4\slash 9$, $p(n,k,d)=(2n-1)^{kd-k}$ satisfies the conditions of Lemma \ref{lem: double edges form a matching} and the conditions of Theorem \ref{thm: Gp has Property (T)}.

As above we may write $\Delta_{k}=\Sigma_{1}\cup\Sigma_{2}\cup\Sigma_{3}$. Now, after collapsing edges, we find $\Sigma_{1}',\Sigma_{3}'$ with the marginal distribution (up to perturbing $p$ to $(1+o_{l}(1))p)$ of 

$$\begin{cases}
\mathfrak{Red}(n,k\slash 3,(2n-1)^{k\slash 3}p):\; k=0\mod 3,\\
\mathfrak{BRed}(n,(k-1)\slash 3,(2n-1)^{(k-1)\slash 3}p):\; k=1\mod 3,\\
\mathfrak{BRed}(n,(k-2)\slash 3,(2n-1)^{(k+1)\slash 3}p):\; k=2\mod 3.\\
\end{cases}$$
Similarly, by collapsing double edges we find $\Sigma_{2}'$ with the marginal distribution of $$\begin{cases}
\mathfrak{Red}(n,k\slash 3,(2n-1)^{k\slash 3}p):\; k=0\mod 3,\\
\mathfrak{Red}(n,(k-1)\slash 3,(2n-1)^{(k+2)\slash 3}p):\; k=1\mod 3,\\
\mathfrak{Red}(n,(k+1)\slash 3,(2n-1)^{(k-2)\slash 3}p):\; k=2\mod 3.\\
\end{cases}$$
 Furthermore, letting $\Sigma' =\Sigma_
{1}'\cup \Sigma_{2}'\cup \Sigma_{3}'$, then as usual we can see that $$\mu_{1}\bigg(D(\Sigma ')^{-1\slash 2}\bigg[ A(\Delta_{k})-A(\Sigma')\bigg]D(\Sigma ')^{-1\slash 2}\bigg)=o_{k}(1).$$ 

By Lemma \ref{lem: spectra of regular reduced random graphs}, there exists some $\delta=o_{k}(1)$ such that  a.a.s.$(k)$: $\Sigma_{2}'$ has a $2(1-\delta)d_{2}$-regular spanning subgraph, $\Pi_{2}$, with $\lambda_{1}(\Pi_{2})>1-o_{k}(1)$; if $k\neq 0\mod 3$ then $\Sigma_{1}',\Sigma_{3}'$ contain $((1-\delta)d_{1},(1-\delta)d_{2})$-regular spanning subgraphs $\Pi_{1},\Pi_{3}$, with 
$\lambda_{1}(\Pi_{1}),\lambda_{1}(\Pi_{3})\geq 1-1\slash (2n-1)+o_{k}(1)$; and if $k= 0\mod 3$ then $\Sigma_{1}',\Sigma_{3}'$ contain $2(1-\delta)d$-regular spanning subgraphs $\Pi_{1},\Pi_{3}$, with 
$\lambda_{1}(\Pi_{1}),\lambda_{1}(\Pi_{3})\geq 1-o_{k}(1).$

As $n\geq 2$, we may apply Lemmas \ref{lem: eigenvalue of union of three graphs on same vertex set} and \ref{cor: eigenvalue of union of three graphs} to deduce that a.a.s.$(k)$: $$\lambda_{1}(\Pi_{1}\cup \Pi_{2}\cup\Pi_{3})>3\slash 4.$$ We see that 
\begin{align*}
        \mu_{1}( A(\Sigma_{1}'\cup\Sigma_{2}'\cup\Sigma_{3}')-A(\Pi_{1}\cup\Pi_{2}\cup\Pi_{3}))&\leq\frac{\delta+o_{k}(1)}{1-\delta}\vert \vert A(\Pi_{1}\cup\Pi_{2}\cup\Pi_{3})\vert \vert_{\infty}\\
        &=o_{k}(1)\vert \vert A(\Pi_{1}\cup\Pi_{2}\cup\Pi_{3})\vert \vert_{\infty}.
    \end{align*}
Hence, letting $\Pi=\Pi_{1}\cup\Pi_{2}\cup\Pi_{3}$, we see that a.a.s.$(k)$:
\begin{align*}
\lambda_{1}(\Sigma ')&=1-\mu_{2}(D^{-1\slash 2}(\Sigma')A(\Sigma')D^{-1\slash 2}(\Sigma'))\\
&=1-\mu_{2}(D^{-1\slash 2}(\Sigma')[A(\Pi)+A(\Sigma ')-A(\Pi)]D^{-1\slash 2}(\Sigma'))\\
&\geq1-\mu_{2}(D^{-1\slash 2}(\Sigma')A(\Pi)D^{-1\slash 2}(\Sigma'))\\
&\;\;-\mu_{1}(D^{-1\slash 2}(\Sigma')(A(\Sigma ')-A(\Pi))D^{-1\slash 2}(\Sigma'))\\
&=1-\bigg(\frac{1}{1-\delta}+o_{k}(1)\bigg)\mu_{2}(D^{-1\slash 2}(\Pi)A(\Pi)D^{-1\slash 2}(\Pi))\\
&\;\;-\bigg(\frac{1}{1-\delta}+o_{k}(1)\bigg)\mu_{1}(D^{-1\slash 2}(\Pi)(A(\Sigma ')-A(\Pi))D^{-1\slash 2}(\Pi))\\
&\geq 1-\dfrac{1}{4}\bigg(\frac{1}{1-\delta}+o_{k}(1)\bigg)-\bigg(\frac{1}{1-\delta}+o_{k}(1)\bigg)\frac{\delta}{1-\delta}\\
&=\dfrac{3[1+o_{k}(1)]}{4}.
    \end{align*}
Since $\lambda_{1}(\Delta_{k})=\lambda_{1}(\Sigma ')+o_{k}(1),$ it follows by Theorem \ref{lem: spectral criterion for Property (T)} that a.a.s.$(k)$ $\Gamma_{k}$  has Property (T). However, as Property (T) is preserved under epimorphisms, it follows immediately that a.a.s.$(k)$ a random group $\Gamma_{k}\sim \Gamma_{p}(n,k,p)$ has Property (T) for any $p$ with $$(2n-1)^{2k\slash 3}p=\Omega_{k}(k).$$
\end{proof}

To prove Theorem \ref{mainthm: property t k-angular model}, we wish to prove the corresponding result for the $k$-angular model: the approach to achieve this is similar.

\begin{lemma}\label{lem: double edges form a matching k fixed}
Let $n\geq 2$, and let $p$ be such that there exists $M\geq 1$ with 
\begin{enumerate}[label=$\roman*)$]
    \item $(2n-1)^{(M+1)l_{k}+L_{k}}p^{M}=o_{n}(1)$,
    \item  $(2n-1)^{2l_{k}+ML_{k}}p^{M}=o_{n}(1)$,
    \item $(2n-1)^{(2M+1)l_{k}+ML_{k}}p^{2M}=o_{n}(1)$,
    \item $(2n-1)^{3Ml_{k}+L_{k}}p^{2M}=o_{n}(1)$,
    \item $(2n-1)^{(M+1)l_{k}+2ML_{k}}p^{2M}=o_{n}(1)$
\end{enumerate}
Let $\Gamma_{k}\sim\Gamma_{p}(n,k,p)$, and let $\Sigma_{i}$ be described as above. For $i=1,2,3$ a.a.s.$(n)$ in $\Sigma_{i}$ there is no pair of vertices $u,v$ with at least $M$ edges between them, and no vertex is connected to more than $M$ other vertices by double edges.
\end{lemma}

\begin{proof}
We first prove this for $i=1,3$. Throughout, note that $k=2l_{k}+L_{k}$. The probability, $\mathbb{P}_{M,1}$, that there exists a pair of vertices $u,v$ with at least $M$ edges between $u$ and $v$ is bounded above by 
\begin{align*}
        \mathbb{P}_{M,1}&\leq O_{n}\left((2n-1)^{l_{k}+L_{k}}(2n-1)^{Ml_{k}}p^{M}\right)\\
        &=O_{n}\left((2n-1)^{(M+1)l_{k}+L_{k}}p^{M}\right)\\
        &=o_{n}(1).
    \end{align*}
The probability, $\mathbb{P}_{doub,1}$ that there are vertices $u\in V_{1}$ and $v_{1},\hdots ,v_{M}\in V_{2}$ with double edges between $u$ and each $v_{i}$ is bounded by
\begin{align*}
        \mathbb{P}_{doub,1}&=O_{n}\left((2n-1)^{l_{k}}(2n-1)^{ML_{k}}(2n-1)^{2Ml_{k}}p^{2M}\right)\\
        &=O_{n}\left((2n-1)^{(2M+1)l_{k}+ML_{k}}p^{2M}\right)\\
        &=o_{n}(1).
    \end{align*}
    
    The probability, $\mathbb{P}_{doub,1}'$ that there are vertices $u\in V_{2}$ and $v_{1},\hdots ,v_{M}\in V_{1}$ with double edges between $u$ and each $v_{i}$ is bounded by
\begin{align*}
        \mathbb{P}_{doub,1}'&=O_{n}\left((2n-1)^{Ml_{k}}(2n-1)^{L_{k}}(2n-1)^{2Ml_{k}}p^{2M}\right)\\
        &=O_{n}\left((2n-1)^{L_{k}+3Ml_{k}}p^{2M}\right)\\
        &=o_{n}(1).
    \end{align*}
    
Let's now switch to $\Sigma_{2}$. Then the probability, $\mathbb{P}_{M,2}$, that there exists a pair of vertices $u,v$ with at least $M$ edges between $u$ and $v$ is bounded above by 
\begin{align*}
        \mathbb{P}_{M,2}&\leq O_{n}\left((2n-1)^{2l_{k}}(2n-1)^{ML_{k}}p^{M}\right)\\
        &=o_{n}(1).
    \end{align*}
Finally, the probability, $\mathbb{P}_{doub,2}$, that there are vertices $u$ and $v_{1},\hdots ,v_{M}$ with double edges between $u$ and each $v_{i}$ is bounded by
\begin{align*}
        \mathbb{P}_{doub,2}&=O_{n}\left((2n-1)^{(M+1)l_{k}}(2n-1)^{2ML_{k}}p^{2M}\right)\\
        &=o_{n}(1).
    \end{align*} 
\end{proof}

\begin{remark} 
Let $d>0$ and $p_{d}=(2n-1)^{kd-k}.$ Then $p_{d}$ satisfies the conditions above for some $M$ if respectively: 

\begin{enumerate}[label=$\roman*)$]
\item $l_{k}+kd-k<0$, so that $d<(l_{k}+L_{k})\slash k$,
\item $L_{k}+kd-k<0$, so that $d<2l_{k}\slash k$,
\item $2l_{k}+L_{k}+2kd-2k<0$, i.e. $d<1\slash 2$ since $2l_{k}+L_{k}=k$,
\item $3l_{k}+2kd-2k<0$, so that 
$d<(k+L_{k}-l_{k})\slash 2k$, and
\item $l_{k}+2L_{k}+2kd-2k<0$, so that $d<(k+l_{k}-L_{k})\slash 2k$.
\end{enumerate}
This reduces to $d<(k-1)\slash 2k.$ For $k\geq 8$, this is satisfied whenever $d<7\slash 16$. For $k\geq 8$, we have $d_{k}\leq 5\slash 12<7\slash 16$, and so we can find $d$ satisfying the requirements of the above lemma and Theorem \ref{thm: Gp has Property (T) n tends to infinity}.
\end{remark}
We can now observe the following.

\begin{theorem}\label{thm: Gp has Property (T) n tends to infinity}
Let $n\geq 2$, $k\geq 8$. Let $p$ be such that 
\begin{align*}
    (2n-1)^{2l_{k}}p&=\Omega_{n}\bigg(\log (2n-1)^{L_{k}}\bigg),\mbox{ and }
     (2n-1)^{l_{k}+L_{k}}p=\Omega_{n}\bigg(\log (2n-1)^{l_{k}}\bigg).  
\end{align*}
Let $\Gamma_{k}\sim\Gamma_{p}(n,k,p)$. Then $\lim_{n\rightarrow \infty }\mathbb{P}(\Gamma_{k}\mbox{ has Property }(T))=1.$
\end{theorem}
We remark that for $d>d_{k}$ and $p=(2n-1)^{kd-k}$ the above is satisfied.

\begin{proof}[Outline of proof of Theorem \ref{thm: Gp has Property (T) n tends to infinity}]
This follows similarly to the proof of Theorem \ref{thm: Gp has Property (T)}. We may assume that $p$ also satisfies the requirements of Lemma \ref{lem: double edges form a matching k fixed} for some $M$. The main replacement is in the fact that the $\Sigma_{i}$ have a very small proportion of disallowed edges so can be treated as having the marginal distribution of an (bipartite) Erd\"os--R\'enyi random graph. We then find regular spanning subgraphs using Lemma \ref{lem: spectra of regular reduced random graphs $l$ growing}, and repeat the above argument, using Lemma \ref{lem: double edges form a matching k fixed} in place of Lemma \ref{lem: double edges form a matching}. This guarantees us that by collapsing double edges, we remove at most $M^{2}$ edges adjacent to each vertex, and the argument follows similarly.
\end{proof}
We then apply the above to prove Theorem \ref{mainthm: property t k-angular model}, as in the case for the density model.
	\bibliographystyle{amsplain}
	\bibliography{bib}

\providecommand{\bysame}{\leavevmode\hbox to3em{\hrulefill}\thinspace}
\providecommand{\MR}{\relax\ifhmode\unskip\space\fi MR }
\providecommand{\MRhref}[2]{%
  \href{http://www.ams.org/mathscinet-getitem?mr=#1}{#2}
}
\providecommand{\href}[2]{#2}
\begin{thebibliography}{10}

\bibitem{antoniuktriangle}
Sylwia Antoniuk, Tomasz \L~uczak, and Jacek \'{S}wi\c{a}tkowski, \emph{Random
  triangular groups at density 1/3}, Compos. Math. \textbf{151} (2015), no.~1,
  167--178. \MR{3305311}

\bibitem{ashcroft2021eigenvalues}
Calum~J Ashcroft, \emph{On the eigenvalues of {E}rdos-{R}enyi random bipartite
  graphs}, arXiv preprint arXiv:2103.07918 (2021).

\bibitem{ashcroftrandom}
Calum~J. Ashcroft and Colva~M. Roney-Dougal, \emph{On random presentations with
  fixed relator length}, Comm. Algebra \textbf{48} (2020), no.~5, 1904--1918.
  \MR{4085767}

\bibitem{Ballmann-Swiatkowski}
W.~Ballmann and J.~\'{S}wi\k{a}tkowski, \emph{On {L2}-cohomology and
  {P}roperty{ (T)} for automorphism groups of polyhedral cell complexes}, Geom.
  Funct. Anal \textbf{7} (1997), no.~4, 615--645.

\bibitem{bekkadelaharpe}
Bachir Bekka, Pierre de~la Harpe, and Alain Valette, \emph{Kazhdan's {P}roperty
  ({T})}, New Mathematical Monographs, vol.~11, Cambridge University Press,
  Cambridge, 2008. \MR{2415834}

\bibitem{chungrandomgraph}
Fan Chung, Linyuan Lu, and Van Vu, \emph{The spectra of random graphs with
  given expected degrees}, Internet Math. \textbf{1} (2004), no.~3, 257--275.
  \MR{2111009}

\bibitem{DRUTU_Mackay}
Cornelia Druţu and John~M. Mackay, \emph{Random groups, random graphs and
  eigenvalues of p-laplacians}, Advances in Mathematics \textbf{341} (2019),
  188--254.

\bibitem{duong}
Yen Duong, \emph{On {R}andom {G}roups: {T}he {S}quare {M}odel at {D}ensity d <
  1/3 and as {Q}uotients of {F}ree {N}ilpotent {G}roups}, ProQuest LLC, Ann
  Arbor, MI, 2017, Thesis (Ph.D.)--University of Illinois at Chicago.
  \MR{3781851}

\bibitem{ferberkirvelevichsudakov}
Asaf Ferber, Michael Krivelevich, and Benny Sudakov, \emph{Counting and packing
  {H}amilton {$\ell$}-cycles in dense hypergraphs}, J. Comb. \textbf{7} (2016),
  no.~1, 135--157. \MR{3436198}

\bibitem{furedikomlos}
Z.~F\"{u}redi and J.~Koml\'{o}s, \emph{The eigenvalues of random symmetric
  matrices}, Combinatorica \textbf{1} (1981), no.~3, 233--241. \MR{637828}

\bibitem{gromovasymptotic}
M.~Gromov, \emph{Asymptotic invariants of infinite groups}, Geometric group
  theory, {V}ol. 2 ({S}ussex, 1991), London Math. Soc. Lecture Note Ser., vol.
  182, Cambridge Univ. Press, Cambridge, 1993, pp.~1--295. \MR{1253544}

\bibitem{hornjohnson}
Roger~A. Horn and Charles~R. Johnson, \emph{Topics in matrix analysis},
  Cambridge University Press, Cambridge, 1994, Corrected reprint of the 1991
  original. \MR{1288752}

\bibitem{kotowski}
Marcin Kotowski and Micha\l Kotowski, \emph{Random groups and {P}roperty
  {$(T)$}: \.{Z}uk's theorem revisited}, J. Lond. Math. Soc. (2) \textbf{88}
  (2013), no.~2, 396--416. \MR{3106728}

\bibitem{mackay_przytycki2015balanced}
John~M. Mackay and Piotr Przytycki, \emph{Balanced walls for random groups},
  Michigan Math. J. \textbf{64} (2015), no.~2, 397--419. \MR{3359032}

\bibitem{Montee_prop_t}
Murphy{K}ate Montee, \emph{Property ({T}) in k-gonal random groups}, arXiv
  preprint arXiv:2104.01621 (2021).

\bibitem{odrsquaremodel}
Tomasz Odrzyg\'{o}\'{z}d\'{z}, \emph{The square model for random groups},
  Colloq. Math. \textbf{142} (2016), no.~2, 227--254. \MR{3418494}

\bibitem{odrcubulatingsquare}
\bysame, \emph{Cubulating random groups in the square model}, Israel J. Math.
  \textbf{227} (2018), no.~2, 623--661. \MR{3846337}

\bibitem{odrzygozdz2019bent}
Tomasz Odrzyg{\'o}{\'z}d{\'z}, \emph{Bent walls for random groups in the square
  and hexagonal model}, arXiv preprint arXiv:1906.05417 (2019).

\bibitem{olliviersharp}
Y.~Ollivier, \emph{Sharp phase transition theorems for hyperbolicity of random
  groups}, Geom. Funct. Anal. \textbf{14} (2004), no.~3, 595--679. \MR{2100673}

\bibitem{Ollivier-Wise}
Y.~Ollivier and D.T. Wise, \emph{Cubulating random groups at density less than
  1/6}, Transactions of the American Mathematical Society \textbf{363} (2011),
  no.~9, 4701--4733.

\bibitem{Ollivier_Jan_Invitation}
Yann Ollivier, \emph{A {J}anuary 2005 invitation to random groups}, Ensaios
  Matem\'{a}ticos [Mathematical Surveys], vol.~10, Sociedade Brasileira de
  Matem\'{a}tica, Rio de Janeiro, 2005. \MR{2205306}

\bibitem{orefactorthm}
Oystein Ore, \emph{Graphs and subgraphs. {II}}, Trans. Amer. Math. Soc.
  \textbf{93} (1959), 185--204. \MR{110650}

\bibitem{shamirupfal}
E.~Shamir and E.~Upfal, \emph{Large regular factors in random graphs},
  Convexity and graph theory ({J}erusalem, 1981), North-Holland Math. Stud.,
  vol.~87, North-Holland, Amsterdam, 1984, pp.~271--282. \MR{791041}

\bibitem{Tuttefactors}
W.~T. Tutte, \emph{Graph factors}, Combinatorica \textbf{1} (1981), no.~1,
  79--97. \MR{602419}

\bibitem{Weyl}
Hermann Weyl, \emph{Das asymptotische {V}erteilungsgesetz der {E}igenwerte
  linearer partieller {D}ifferentialgleichungen (mit einer {A}nwendung auf die
  {T}heorie der {H}ohlraumstrahlung)}, Math. Ann. \textbf{71} (1912), no.~4,
  441--479. \MR{1511670}

\bibitem{zuk1996}
A.~\.{Z}uk, \emph{La propri{\'e}t{\'e} {(T)} de {K}azhdan pour les groupes
  agissant sur les polyedres}, CR Acad. Sci. Paris S{\'e}r. I Math.
  \textbf{323} (1996), 453--458.

\bibitem{Zuk}
\bysame, \emph{{P}roperty {(T)} and {K}azhdan constants for discrete groups},
  Geom. Funct. Anal. \textbf{13} (2003), no.~3, 643--670.

\end{thebibliography}
	{\sc{DPMMS, Centre for Mathematical Sciences, Wilberforce Road, Cambridge, CB3 0WB, UK}}\\
	\emph{E-mail address}: cja59@dpmms.cam.ac.uk
\end{document}